\documentclass[11pt]{article}

\usepackage{amsmath,amssymb,amsthm,mathtools,amsfonts,xcolor,enumitem,bbm}
\usepackage{authblk,mathrsfs,mathtools}
\usepackage{algorithm}
\usepackage{hyperref}
\usepackage{algpseudocode}
\usepackage{graphicx, float}
\usepackage{graphicx}
\usepackage{subcaption}
\usepackage[margin=1in]{geometry}
\usepackage{setspace}
\numberwithin{equation}{section}

\newtheorem{assumption}{Assumption}[section]
\newtheorem{theorem}{Theorem}[section]
\newtheorem{lemma}{Lemma}[section]
\newtheorem{corollary}{Corollary}[section]
\newtheorem{definition}{Definition}[section]
\newtheorem{proposition}{Proposition}[section]
\theoremstyle{remark}
\newtheorem{remark}{Remark}[section]

\newcommand{\F}{\mathcal{F}}

\newcommand{\E}{\mathbb{E}}

\newcommand{\R}{\mathbb{R}}

\newcommand{\KL}{\mathrm{KL}}
\newcommand{\tr}{\mathrm{tr}}

\newcommand{\Law}{{\rm Law}}
\renewcommand{\P}{\mathbb{P}}

\providecommand{\keywords}[1]
{
  \small	
  \textbf{\textbf{Keywords---}} #1
}

\usepackage{geometry}
\usepackage{verbatim}
\renewcommand{\P}{\mathbb{P}}

\theoremstyle{plain}

\title{Generalized specific entropy on Wiener space with application to Martingale Optimal Transport}

\author[1]{Francois Buet-Golfouse}
\author[1,2]{Anaïs Després}
\author[2]{Zhenjie Ren}
\author[3]{Xin Zhang \thanks{X. Zhang is partially supported by the NSF Grant DMS-2508556.}}

\affil[1]{AIML Global Markets, Barclays}
\affil[2]{LaMME, Université Évry Paris-Saclay, Évry-Courcouronnes, France}
\affil[3]{FRE Department, New York University, New York, USA}

\begin{document}
\maketitle

\begin{abstract}
% Classical entropy regularization is poorly suited to continuous-time martingale transport, since relative entropy between diffusion laws typically forces the volatility characteristics to coincide. We introduce a new specific-entropy framework to overcome this problem by approximating continuous martingales with Poisson jump processes. In the Gaussian-mark case, this yields explicit generalized specific entropy functionals on Wiener space. A key feature is that the limiting cost is not intrinsic to the limiting martingale laws alone, but also reflects the microscopic mechanism by which the continuous paths are approximated. Our method avoids the grid-refinement strategy of previous approaches, which leads to increasingly high-dimensional multimarginal Sinkhorn problems. It also avoids the uniform time-grid assumption underlying these approaches, which implies that price changes occur at fixed intervals regardless of volatility—an assumption that is inconsistent with observed market activity.
% We prove weak convergence of the Poisson approximations and derive explicit limiting entropy functionals. For a particular Poisson scheme, the resulting cost  leads to a continuous-time specific-entropic martingale optimal transport problem, called SEMOT. The chosen cost provides the compactness needed for existence and strong duality, and leads formally to a coupled Hamilton–Jacobi–Bellman/Fokker–Planck system. This structure suggests Sinkhorn-type numerical schemes, which are implemented in one and two dimensions.

Classical entropy regularization is poorly suited to continuous-time martingale transport, since relative entropy between diffusion laws typically forces their volatility characteristics to coincide. We introduce a specific-entropy framework based on Poisson jump approximations of continuous martingales. In the Gaussian-mark case, this yields explicit generalized specific entropy functionals on Wiener space, whose limiting costs depend not only on the limiting martingale laws but also on the microscopic approximation mechanism. This Poissonization approach avoids deterministic grid refinement and the associated high-dimensional multimarginal Sinkhorn problems, while allowing jump intensities to reflect local volatility.

We prove weak convergence of the Poisson approximations and identify the limiting entropy functionals. For a trace-normalized Poisson scheme, the resulting cost defines a continuous-time specific-entropic martingale optimal transport problem, called SEMOT. This cost yields compactness, existence, and strong duality, and leads formally to a coupled Hamilton-Jacobi-Bellman/Fokker-Planck system. The resulting structure suggests Sinkhorn type numerical schemes, which we implement in one and two dimensions.

\end{abstract}
\keywords{Entropy, Martingale optimal transport, Poisson approximation, Duality.}

\section{Introduction}

% {\color{blue}The introduction overall looks good to me. Here are some comments. 

% 1. I think the intro to OT and EOT is a bit long. Notation $OT_c$ $EOT_c$ are not used anymore later in the paper. Perhaps we could start from the definition relative entropy on $P(X)$ where $X$ is an arbitrary polish space. Say it is not well-posed for processes with different volatilities, and then very quickly introduce specific  entropies. 

% 2. We could mention the purpose of constructing specific entropies is two-fold: one is divergence between martingale laws; the other is to use in MOT problem (we need some brief intro to MOT and references).

% 3. It would be better to emphasize some advantages of SEMOT compared with MOT, e.g. regularization effect? I am not sure about this point. 

% }

Optimal transport provides a flexible framework for comparing probability measures and constructing couplings between them. Given two probability measures \(\mu_0,\mu_1\) on \(\R^d\) and a measurable cost
$
c:\R^d\times\R^d\to \R\cup\{+\infty\},
$
the classical Kantorovich problem is
\begin{equation}\label{eq:OT_intro}
\inf_{\pi\in\Pi(\mu_0,\mu_1)}
\int_{\R^d\times\R^d} c(x,y)\,\pi(dx,dy),
\end{equation}
where \(\Pi(\mu_0,\mu_1)\) denotes the set of couplings between \(\mu_0\) and \(\mu_1\). Beyond its original geometric formulation, optimal transport has become a central tool in probability, analysis, PDEs, statistics, and machine learning; see for instance \cite{villani2009optimal,santambrogio2015optimal}.

A particularly important regularization of \eqref{eq:OT_intro} is entropic optimal transport (EOT), where one penalizes deviations from a reference coupling \(R\) by relative entropy $H$:
\[
\inf_{\pi\in\Pi(\mu_0,\mu_1)}
\left\{
\int c(x,y)\,\pi(dx,dy)+\varepsilon H(\pi\,|\,R)
\right\}.
\]
This formulation is appealing both theoretically and computationally. It yields a strictly convex problem, is closely related to the Schrödinger bridge problem, and can be solved efficiently by the celebrated Sinkhorn's algorithm and its variants; see, e.g., \cite{leonard2014survey,peyre2019computational,galichon2016optimal,benamou2021optimal}.
However, standard entropic regularization is naturally adapted to perturbations of the drift around the fixed reference $R$. In continuous time, relative entropy with respect to a diffusion reference is finite only when the candidate process has the same diffusion coefficient as the reference. As a result, the classical Schrödinger/EOT framework is not well suited to problems in which the volatility is itself the main control variable. This limitation is particularly severe in applications from mathematical finance, such as model calibration and robust pricing, where one seeks martingale models consistent with observed option prices and where the relevant degree of freedom is the quadratic variation rather than the drift; see for example \cite{avellaneda1997calibrating,henrylabordere2019martingale,benamou2024entropictransport,beiglbock2013model,henrylabordere2017modelfree}. 

This issue appears prominently in martingale optimal transport (MOT) \cite{henrylabordere2017modelfree, BeNuTo16, HuTr19}. In its static form, MOT replaces \(\Pi(\mu_0,\mu_1)\) by the set
\[
\mathcal M(\mu_0,\mu_1)
:=
\left\{
\pi\in\Pi(\mu_0,\mu_1): \E_\pi[Y\mid X]=X \ \pi\text{-a.s.}
\right\},
\]
and minimizes the same transport cost over martingale couplings.
In this MOT setting, a natural idea is to replace the classical entropy penalty with a divergence that remains meaningful when volatilities differ. A first answer comes from time discretization. When two continuous martingale laws are observed on a discrete time grid, the corresponding discrete laws may have finite relative entropy even though the original continuous-time laws are mutually singular. After suitable rescaling, this discrete entropy may converge to a nontrivial limit, known as the specific relative entropy (SRE) which was first introduced in \cite{gantert1991largedeviations,backhoffveraguas2023specific}. Formally, if $X$ denotes the canonical process on $C([0,T];\R^d)$ and ${\bf X}^n:=(X_{t_1},\cdots, X_{t_n})$
denotes the a uniform time discretization of $X$, SRE is defined as 
\begin{equation}
\label{eq:SRE}
h(\mathbb Q\,|\,\mathbb P)
:=
\lim_{n\to\infty}\frac1n
H\!\left({\bf X}^n\#\mathbb Q\,\middle|\, {\bf X}^n\#\mathbb P\right).
\end{equation}

This idea has recently led to entropy-regularized versions of MOT based on uniform time discretization. In these approaches (\cite{benamou2024entropic,benamou2024entropictransport}), one first discretizes time, studies the corresponding multimarginal transport problems, proves duality and convergence of the numerical method at the discrete level, and then lets the mesh go to zero. However, the duality theory and the numerical methods are established for the discretized problems, not for the original continuous-time problem. In particular, the resulting algorithms are based on solving a sequence of increasingly high-dimensional multimarginal Sinkhorn problems, which rapidly become expensive as the time grid is refined. A more subtle issue lies in the use of uniform time discretization. While seemingly innocuous, this choice implicitly encodes an assumption on the underlying microscopic market model—namely, that price changes (or equivalently, trades) occur at a constant frequency, regardless of the instantaneous volatility. This stands in contrast with previous studies \cite{wyart2008spread, pohl2018trading,muhlekarbe2026unified} , which suggest that trading intensity is closely linked to volatility dynamics.

The central objective of the present paper is precisely to formulate and analyze a continuous-time specific entropic martingale optimal transport problem, which we call SEMOT.  Our goal is to define a continuous-time entropy cost that makes the problem well posed but flexible enough to allow martingale volatility control, and tractable enough to lead to a direct dual formulation and a continuous-time numerical scheme. Rather than starting from deterministic time discretizations, we study relative entropy through Poissonization. More precisely, we approximate continuous martingale laws by families of pure-jump martingale laws and investigate the asymptotic behavior of the entropy between these jump models. This point of view reveals a broader class of limiting entropy costs than those obtained from uniform time grids which corresponds to SRE. It shows, in particular, that the entropy limit is not solely determined by the limiting continuous martingale law: it also depends on how the approximation is performed at the microscopic level. This flexibility is not a drawback but a key advantage: mathematically, among all the possible entropy limits, one can select costs with the most suitable structural properties needed for duality. From a financial perspective, the Poissonization viewpoint makes it possible to encode non-uniform
trading activity at the approximation level, and in particular to consider
schemes in which the jump intensity depends on the local volatility. Thus, the
microscopic frequency of price changes need not be fixed exogenously, but can
adapt to the activity of the underlying martingale model. Potential applications are probably not limited to finance: for instance, our approach
could improve on the recent OT modeling for genes’ developmental trajectories of cells as in \cite{bunne2024singlecellomics}.

There are relatively few works combining specific entropy, martingale transport, and continuous-time numerical methods. The discrete specific entropy viewpoint is developed in \cite{backhoffveraguas2023specific}, while related divergences for martingale models appear in works such as \cite{avellaneda1997calibrating,backhoffveraguas2025specificwasserstein}. The use of specific entropy as a regularization for model calibration and martingale transport is explored in \cite{benamou2024entropic,benamou2024entropictransport}.

The contribution of this paper is to develop a Poissonization-based framework for specific entropic martingale optimal transport. First, we establish weak convergence of a broad class of Poisson approximations to continuous martingale laws, and, more importantly, identify the scaling limits of relative entropy associated with different Gaussian-mark Poissonization schemes. This leads to generalized specific relative entropy functionals on Wiener space and shows that the limiting entropy depends not only on the limiting martingale laws, but also on the microscopic approximation mechanism. For a particular choice of Poissonization, we show that, when d=1, the resulting cost coincides with the specific relative entropy characterized in \cite[Theorem 1]{BaUn23}. Second, we focus on the trace-normalized Poissonization, which is better aligned with microscopic market observations linking trading intensity and volatility, and  provides the a priori estimates needed to prove tightness, lower semicontinuity, existence, and strong duality for the corresponding continuous-time SEMOT problem. We also give a large deviation interpretation of  the specific entropy functional.
Finally, we derive a coupled HJB--Fokker--Planck system characterizing the optimal plan of SEMOT problem and propose a Sinkhorn-type iterative algorithm. %Unlike discretize-then-refine approaches, these algorithms act directly on the continuous-time dual problem associated with SEMOT.

The remainder of the paper is organised as follows. In Section 2, we introduce the continuous martingale framework and the Poisson approximations that serve as the target and reference models. Section 3 contains the main theoretical results: weak convergence of the Poisson approximations, identification of the scaling limit of the normalized relative entropy in the Gaussian-mark setting, and the application to strong duality for SEMOT. In Section 4, we present the numerical results. Section 5 is devoted to the proofs of the main results.

\section{Preliminaries}
\label{sec:Preliminaries}

\paragraph{Notation}
Fix $d\in\mathbb N^*$ and the time horizon $[0,1]$. For $x,y\in\mathbb R^d$, let $x\cdot y$ and $|x|$ denote the Euclidean inner product and norm. For a square matrix $A$, write $A^\top$, $\tr(A)$, and $\det(A)$ for its transpose, trace, and determinant, and let $I_d$ be the identity matrix. We denote by $\mathbb S^d$, $\mathbb S^d_+$, and $\mathbb S^d_{++}$ the sets of symmetric, nonnegative symmetric, and positive definite $d\times d$ matrices, respectively; for $A,B\in\mathbb S^d$, $A\preceq B$ means $B-A\in\mathbb S^d_+$.

The canonical space is either $\Omega=C([0,1];\mathbb R^d)$ or $\Omega=D([0,1];\mathbb R^d)$, with canonical process $X_t(\omega)=\omega_t$ and canonical filtration $\mathbb F=(\mathcal F_t)_{0\le t\le1}$; in the càdlàg case, $\mu(dt,dz)$ denotes the jump measure of $X$. For a probability measure $\mathbb P$ on $\Omega$, we write $\E^\mathbb P$ for expectation and $\mathcal L_\mathbb P(Y)$ for the law of $Y$, omitting $\mathbb P$ when clear from the context. We denote by $\mathbb W$ the Wiener measure on $C([0,1];\mathbb R^d)$.

Let \(E\) be a topological space. We denote by \(\mathcal M(E)\) the space
of finite nonnegative Borel measures on \(E\), and by
$
\mathcal M_1(E)
:= \left\{ \mu \in \mathcal M(E) : \mu(E)=1 \right\}
$
the space of Borel probability measures on \(E\). We also denote by
\(\overline{\mathcal M}(E)\) the space of finite signed Borel measures on
\(E\), endowed with the weak topology
$
\sigma\bigl(\overline{\mathcal M}(E), C_b(E)\bigr).
$

For probability measures $\mathbb P,\mathbb Q$, define
$H(\mathbb P\| \mathbb Q):=\E_\mathbb P\!\left[\log\left(\frac{d\mathbb P}{d\mathbb Q}\right)\right]$ if $\mathbb P\ll \mathbb Q$, and $+\infty$ otherwise. If \(T : E \to F\) is a measurable map and \(\mu\) is a measure on \(E\), we denote by
\(T_{\#}\mu\) the pushforward of \(\mu\) by \(T\), defined by
$
(T_{\#}\mu)(B) = \mu\bigl(T^{-1}(B)\bigr)
\ \text{for every measurable set } B \subseteq F.
$

In the continuous setting $\Omega=C([0,1];\mathbb R^d)$, let $\mathcal P$
be the set of probability measures $\mathbb P$ under which $X$ is a continuous local martingale admitting the representation
$
X_t = X_0 + M_t^{\mathbb P},\ \mathbb P\text{-a.s. for all } t\in[0,1],
$
where $M^{\mathbb P}=(M_t^{\mathbb P})_{0\le t\le1}$ is a continuous local martingale satisfying
$M_0^{\mathbb P}=0$.
We also denote by
$
A_t^{\mathbb P} := \langle M^{\mathbb P}\rangle_t, \
A^{\mathbb P}=(A_t^{\mathbb P})_{0\le t\le1},
$
the quadratic variation of $X$ under $\mathbb P$.
Throughout this section, we restrict attention to probability measures $\mathbb P$ for which
the process $A^{\mathbb P}$ is absolutely continuous with respect to the Lebesgue measure,
$\mathbb P$-a.s. In this case, there exists an $\mathbb F$-progressively measurable process
$
\Sigma^{\mathbb P}=(\Sigma_t^{\mathbb P})_{0\le t\le1}, \
\Sigma_t^{\mathbb P}\in\mathbb{S}_d^{+},
$
such that
$
A_t^{\mathbb P} = \int_0^t \Sigma_s^{\mathbb P}\,ds,
\ \mathbb P\text{-a.s. for all } t\in[0,1].
$
The process $\Sigma^{\mathbb P}$ is referred to as the instantaneous covariance characteristic
of $X$ under $\mathbb P$.
\noindent
For \(\mu_0,\mu_1\in\mathcal M(\mathbb R^d)\), we define
$
\mathcal P(\mu_0):=\{\mathbb P\in\mathcal P:\mathbb P\circ X_0^{-1}=\mu_0\},
$ and
$
\mathcal P(\mu_0,\mu_1):=\{\mathbb P\in\mathcal P(\mu_0):\mathbb P\circ X_1^{-1}=\mu_1\}.\\
$

We start this section by recalling the martingale-problem formulation used to characterize
the limiting continuous martingales and, subsequently, their Poisson
approximations.
\begin{definition}
\label{def:target measure}
 Given $\Sigma(t,x)  \in \mathbb S^d_{+}$, a probability measure \(\mathbb P\) on the c\`adl\`ag path space
\(D([0,1],\mathbb R^d)\), equipped with its Borel \(\sigma\)-algebra,
is said to solve the martingale problem for the time-dependent generator
\begin{equation}\label{eq:mtgp_d}
\mathcal L_t\varphi(x)
=
\frac12 \tr\!\big(\Sigma(t,x)D^2\varphi(x)\big),
\qquad (t,x)\in[0,1]\times\mathbb R^d,
\end{equation}
with initial distribution \(\delta_{x_0}\) if, denoting by \((X_t)\) the canonical process,
\begin{enumerate}
    \item[{\rm (i)}] \((X_0)_{\#}\mathbb P=\delta_{x_0}\);
    \item[{\rm (ii)}] for every \(\varphi\in C_b^2(\mathbb R^d)\), the process
    \[
    M_t^\varphi
    :=
    \varphi(X_t)-\varphi(X_0)-\int_0^t \mathcal L_s\varphi(X_s)\,ds
    \]
    is a martingale.
\end{enumerate}
\end{definition}

Given two covariance fields \(\Sigma_1\) and \(\Sigma_2\), we denote by
\(\mathbb P\) and \(\mathbb Q\) the corresponding target and reference laws,
namely the solutions, whenever they exist, of the martingale problems associated
with \(\Sigma_1\) and \(\Sigma_2\).
Our goal is to approximate these continuous martingale laws by pure-jump
martingale laws \((\mathbb P^n)_{n\ge1}\) and \((\mathbb Q^n)_{n\ge1}\) such
that
\[
\mathbb P^n \Longrightarrow \mathbb P,
\qquad
\mathbb Q^n \Longrightarrow \mathbb Q,
\qquad n\to\infty.
\]

Let us consider an approximating jump process of intensity $n\lambda$ and covariance marks $\frac{\bar\Sigma}{n}$ for which we introduce the following jump generator. For every \(t\in[0,1]\),
\begin{equation}\label{eq:L_n}
(\mathcal L_t^n \varphi)(x)
=
n \lambda(t,x)
\int_{\mathbb{R}^d}
\left[
\varphi\!\left(x+\frac{1}{\sqrt n}\bar\Sigma(t, x)^{1/2}e\right)-\varphi(x)
\right]\eta(de).
\end{equation}
Here \(\lambda:[0,1]\times\mathbb R^d\to[0,\infty)\) is a measurable function describing the local jump intensity, while
\(\bar\Sigma:[0,1]\times\mathbb R^d\to\mathbb S_+^d\) specifies the covariance  of the jump marks and $\eta$ is a probability measure on $\mathbb R^d$.
For each \(n\ge 1\), let \(\mathbb P^n\) be a solution to the martingale problem
on \(D([0,1];\mathbb R^d)\) associated with
\((\mathcal L_t^n)_{0\le t\le 1}\) and initial distribution \(\delta_{x_0}\). The main convergence statement proved below will show, under suitable
assumptions on \(\lambda\), \(\bar\Sigma\), and \(\eta\), that
$
\mathbb P^n \Longrightarrow \mathbb P$
as $n\to\infty,
$
where \(\mathbb P\) solves the continuous martingale problem with covariance
$
\Sigma_1=\Sigma:=\lambda\bar\Sigma.
$
The reference law \(\mathbb Q\) is approximated in the same way.

\section{Main Results}

This section contains the main contribution of the paper. We first establish a general weak convergence result for general Poisson approximations of continuous martingales. We then identify the limit of the normalized relative entropy in the Gaussian-mark setting. Finally, we apply this analysis to a particular superlinear Poissonization  scheme in order to obtain a strong duality result for the associated MOT problem.

\subsection{Weak Convergence} 
We first state the assumption under which the convergence result will hold. It guarantees that the  martingale problem in Definition \ref{def:target measure} admits a unique solution and provides the integrability needed to prove tightness and to control the generator approximation uniformly along the sequence. 

\begin{assumption}
\label{ass:weak_conv}
Suppose $\varepsilon\in(0,1)$, and the following hold.
\begin{enumerate}
 \item[{\rm (i)}] The measure \(\eta\) satisfies
    \[
    \int_{\mathbb{R}^d} e\,\eta(de)=0,
    \qquad
    \int_{\mathbb{R}^d} ee^\top\,\eta(de)=I_d,
    \qquad
    \int_{\mathbb{R}^d} |e|^{2+\varepsilon}\,\eta(de)<\infty.
    \]

    \item[{\rm (ii)}]The map \(\Sigma:=\lambda\bar\Sigma:[0,1]\times\mathbb R^d\to\mathbb S_{++}^d\) is continuous and uniformly elliptic, that is, there exist constants \(0<m\le M<\infty\) such that
     \[
    m I_d\preceq \Sigma(t,x)\preceq M I_d,
    \qquad \forall (t,x)\in[0,1]\times\mathbb R^d.
    \]

    \item[{\rm (iii)}] Let
    \[
    \mathcal M
    :=
    \left\{
        \mathbb P^n \in\mathcal M_1(D([0,1],\mathbb R^d))
        \;\middle|\;
        \mathbb P^n \text{ solves the martingale problem for } \mathcal L^n,\ n\in\mathbb N
    \right\}.
    \]
    Then
    \[
    \sup_{\mathbb P^n \in \mathcal M}
    \mathbb E^{\mathbb P^n}\!\left[
        \int_0^1
        \tr(\bar\Sigma(t, X_t))^{1+\varepsilon/2}\,
        \lambda(t, X_t)\,dt
    \right]
    < \infty.
    \]
\end{enumerate}
\end{assumption}

We have the following convergence result. 

\begin{theorem}
\label{thm:weakconvergence d dim}
Under Assumption
\ref{ass:weak_conv},
 \(\mathbb P^n\) converges to $\mathbb P$ in weak topology as $n \to \infty$.
\end{theorem}

% The proof proceeds in three steps. First, we establish tightness of the sequence
% \((\mathbb P^n)_{n\ge1}\) on \(D([0,1];\mathbb R^d)\). Second, we show that any subsequential limit is in fact supported on continuous paths, so that limit points belong to \(C([0,1];\mathbb R^d)\). Third, we identify every such limit point by proving that it solves the martingale problem associated with the diffusion generator \(\mathcal L\).

% \medskip

% \noindent
% We begin with tightness.

% \noindent
% Tightness alone is not sufficient, since a priori a limit point could still have jumps. The next lemma rules this out by showing that the jump sizes become negligible as \(n\to\infty\).

% \noindent
% As a consequence, every weak limit point of \(\Law(X)\) is concentrated on continuous paths. This reduces the identification problem to continuous-path laws only.

% \medskip

% \noindent
% To characterize such limit points, we compare the jump generator \(\mathcal L^n\) with the limiting diffusion generator \(\mathcal L\). The following estimate shows that \(\mathcal L^n\) converges uniformly to \(\mathcal L\) on smooth test functions, with an explicit rate.

% \noindent
% This estimate is the key analytical input for passing to the limit in the martingale problem. It allows us to transfer the martingale property from the approximating jump processes to any weak limit.

\subsection{Scaling limit of relative entropy}
\label{sec:limit_of_chaos}
Now, let us choose a particular Poissonization where the marks are Gaussian. In this setting, we can derive an exact finite-\(n\) formula for \(H(\mathbb P_1^n\|\mathbb P_2^n)\), and then pass to the limit to identify the corresponding rate functional. We will show that the entropy asymptotics depend on the Poissonization scheme.

For each \(i\in\{1,2\}\), let \(\mathbb P_i^n\) be the law on the canonical space under which
the canonical process \(X\) solves the martingale problem with time-dependent generator
\[
(\mathcal L_{i,t}^n f)(x)
=
\lambda_i^n(t,x)\int_{\mathbb R^d}
\bigl(f(x+z)-f(x)\bigr)\,\eta_i^n(t,x,dz),
\]
for all \(f\in C_b^2(\mathbb R^d)\), where
$
\lambda_i^n(t,x)=n\lambda_i(t,x),
\
\eta_i^n(t,x,dz)
=
\mathcal N\!\left(0,\frac{\bar\Sigma_i(t,x)}{n}\right)(dz).
$
Equivalently, if \(\mu(dt,dz)\) denotes the jump measure of \(X\), then under \(\mathbb P_i^n\),
its predictable compensator is given by
$
\nu_i^n(dt,dz)
=
\lambda_i^n(t,X_{t-})\,\eta_i^n(t,X_{t-},dz)\,dt.
$
In particular, \(\mathbb P_1^n\) is the target measure and \(\mathbb P_2^n\) is the reference measure.
For each $(t,x)\in[0,1]\times\mathbb{R}^d$, we also define
\[
\begin{aligned}
\ell(t,x)
:={}&
\lambda_1(t,x)\log\frac{\lambda_1(t,x)}{\lambda_2(t,x)}
-\lambda_1(t,x)
+\lambda_2(t,x) \\
&+
\frac{\lambda_1(t,x)}{2}
\left[
\operatorname{tr}\!\bigl(\bar{\Sigma}_2^{-1}(t,x)\bar{\Sigma}_1(t,x)\bigr)
-d
-\log\det\!\bigl(\bar{\Sigma}_2^{-1}(t,x)\bar{\Sigma}_1(t,x)\bigr)
\right].
\end{aligned}
\]

\begin{assumption}
\label{ass:chaos_limit}
The following hold:
\begin{itemize}
    \item[{\rm (i)}] The predictable compensator of the jump measure under \(\mathbb P_1^n\) is absolutely continuous with respect to that under \(\mathbb P_2^n\), i.e.
    $
    \nu_1^n(\omega;dt,dz)\ll \nu_2^n(\omega;dt,dz)
    \ \text{for \(\mathbb P_2^n\)-a.e. \(\omega\)}.$

    \item[{\rm (ii)}] The functions \(\lambda_1,\lambda_2:[0,1]\times\mathbb R^d\to(0,\infty)\) and
    $
    \bar\Sigma_1,\bar\Sigma_2:[0,1]\times\mathbb R^d\to S^d_{++}
    $
    are continuous.

    \item[{\rm (iii)}]
    For $\omega\in D([0,1],\R^d)
    $, the functional $\Psi(\omega):=\int_0^1 \ell(t,\omega_t)\,dt$ is assumed to be uniformly integrable with respect to the sequence of measures \((\mathbb P_1^n)_{n\ge1}\), in the sense that
    \[
    \lim_{K\to\infty}\sup_{n\ge1}
    \mathbb E^{\mathbb P_1^n}\!\left[
    |\Psi|\,\mathbf 1_{\{|\Psi|>K\}}
    \right]=0.
    \]

\end{itemize}
\end{assumption}

\begin{theorem}[Limit of the normalized relative entropy]
\label{thm:limit of chaos}
Under Assumptions \ref{ass:weak_conv} (ii) and \ref{ass:chaos_limit},
\begin{equation}\label{eq:se}
\lim_{n\to\infty}\frac{1}{n}H(\mathbb P_1^n\|\mathbb P_2^n)
=
\mathbb{E}^{\mathbb{P}}\!\left[\int_0^1 \ell(t,X_t)\,dt\right].
\end{equation}
\end{theorem}

\begin{remark}
Although Assumption~\ref{ass:chaos_limit} (iii) may appear abstract, it is easy to verify in practice. For instance, it holds whenever the functions $\lambda_i$ and $\bar\Sigma_i$, $i=1,2$, are uniformly bounded above and bounded away from zero.
\end{remark}
\medskip
\begin{remark}\label{rmk:twoschemes}
The general limit formula above applies to a broad class of Gaussian-mark Poissonizations. We now illustrate it on two representative schemes. 

\begin{enumerate}
    \item[{\rm (i)}]
The first scheme modifies only the mark distribution. For \(i=1,2\), let
$
\lambda_i^n(t,X_{t-})=n,
$
and define the mark kernels by
$
\eta_i^n(t,x,dz)
=
\mathcal N\!\left(0,\frac{\bar\Sigma_i(t,x)}{n}\right)(dz).
$
Under Assumption \ref{ass:weak_conv} (ii), \(\mathbb P_2^n\Rightarrow\mathbb Q\) and \(\mathbb P_1^n\Rightarrow\mathbb P\). Moreover, under conditions of Theorem \ref{thm:limit of chaos},
\begin{equation*}\label{eq:KL-limit-B-hatSigma}
\frac{1}{n}H(\mathbb P_1^n\|\mathbb P_2^n)
\longrightarrow
\mathbb E\!\left[\int_0^1
\frac12
\left(
\operatorname{tr}\!\bigl(\bar\Sigma_2(t,X_t)^{-1}\bar\Sigma_1(t,X_t)\bigr)
-d
-\log\det\!\bigl(\bar\Sigma_2(t,X_t)^{-1}\bar \Sigma_1(t,X_t)\bigr)
\right)\,dt\right].
\end{equation*}
For \(d=1\), this coincides with the relative specific entropy identified in \cite[Theorem 1]{BaUn23}.

\item[{\rm (ii)}]  The second scheme splits the covariance mismatch into a scalar part, encoded in the jump intensity, and a normalized matrix part, encoded in the mark distribution. Assume that \(\Sigma_1(t,x)\) and \(\Sigma_2(t,x)\) are positive definite for all \((t,x)\), and write
\begin{equation}\label{eq:ref_volatility}
\Sigma_2(t,x)=\lambda_2(t,x)\,\bar{\Sigma}_2(t,x),
\qquad
\lambda_2(t,x)>0,
\qquad
\bar{\Sigma}_2(t,x)\in\mathbb S_{++}^d.
\end{equation}
Define
\begin{equation}\label{eq:target_measure}
\begin{aligned}
\lambda_1(t,x)
:=
\frac{1}{d}\,\operatorname{tr}\!\bigl(\bar{\Sigma}_2(t,x)^{-1}\Sigma_1(t,x)\bigr),
\qquad
\bar{\Sigma}_1(t,x)
:=
\frac{\Sigma_1(t,x)}{\lambda_1(t,x)}.
\end{aligned}
\end{equation}
so that
$
\operatorname{tr}\!\bigl(\bar{\Sigma}_2(t,x)^{-1}\bar{\Sigma}_1(t,x)\bigr)=d.
$
Then set
\[
\lambda_i^n(t,X_{t-})=n\,\lambda_i(t,X_{t-}),
\qquad
\eta_i^n(t,x,dz)
=
\mathcal N\!\left(0,\frac{\bar{\Sigma}_i(t,x)}{n}\right)(dz),
\qquad i=1,2.
\]
Here again, \(\mathbb P_1^n\Rightarrow\mathbb P\) and \(\mathbb P_2^n\Rightarrow\mathbb Q\), and
\[
\begin{aligned}
\frac{1}{n}H(\mathbb P_1^n\|\mathbb P_2^n)
\longrightarrow
\mathbb E\!\Biggl[\int_0^1
\Biggl(
&\lambda_1(t,X_t)\log\frac{\lambda_1(t,X_t)}{\lambda_2(t,X_t)}
-\lambda_1(t,X_t)
+\lambda_2(t,X_t) \\
&\quad
-\frac{\lambda_1(t,X_t)}{2}
\log\det\!\bigl(
\bar{\Sigma}_2(t,X_t)^{-1}
\bar{\Sigma}_1(t,X_t)
\bigr)
\Biggr)\,dt\Biggr].
\end{aligned}
\]
% In particular, when \(\lambda_1=\lambda_2\equiv1\), this reduces to the specific relative entropy and corresponds to time-space homogeneous discretizations.
The trace normalization is not merely a technical choice; it also has a natural financial interpretation. Although one could instead normalize by the determinant, the trace is better suited to an interpretation in terms of aggregate market activity. Indeed, \cite{wyart2008spread} argues that short-horizon volatility is strongly linked to trading activity, with volatility per trade being largely generated by trading itself. Scaling relations linking trading activity to volatility are also derived in \cite{pohl2018trading}. Thus, if volatility is viewed as a proxy for market activity, then the total activity in a multi-asset model should be obtained by adding the contributions of the individual assets, rather than multiplying them.
\end{enumerate}

\noindent
Both schemes converge to the same continuous martingale limit, but their KL asymptotics differ
because they allocate the mismatch between $\Sigma_1$ and $\Sigma_2$ differently. The KL typically grows as
$H(\mathbb P_1^n\|\mathbb P_2^n)\sim n\times(\text{rate functional})$, with different rate functionals
depending on the Poissonization scheme.
\end{remark}

% \begin{remark}
% The trace normalization admits the most natural financial interpretation. 
% According to \cite{wyart2008spread}, short-horizon volatility is strongly linked to trading activity, in the sense that volatility per trade is largely generated by trading itself. \cite{pohl2018trading} also derives scaling relations linking trading activity to volatility.
% If volatility is viewed as a proxy for market activity, then in a multi-asset model the overall activity should be obtained by adding the contributions of the different assets, rather than multiplying them if we use determinant for example.
% \end{remark}

In \cite{Ga91}, Gantert showed that the specific relative entropy is a rate function for a large deviation principle. Here we observe that, the reciprocal specific relative entropy, a special case of \eqref{eq:se}, is also related to a large deviation principle. 
\noindent
Denote by $\omega$ a real-valued sample path with  $\omega(0)=0$. Let us define, for each $ n \in \mathbb N$,
\begin{align*} 
S_k^n(\omega)&=  \sum_{i=1}^k \left(\omega\left(\frac{i}{n}\right) -\omega\left(\frac{2i-1}{2n}\right)\right)^2+\left(\omega\left(\frac{2i-1}{2n}\right) -\omega\left(\frac{i-1}{n}\right)\right)^2, \quad k=1,\dotso n, 
\end{align*}
and counting processes
\begin{align*}
N^n_t(\omega)&=\max \{k/n : \, k \in \mathbb N, \  S_k^n(\omega) \leq t   \}, \ t \in [0,1].
\end{align*}
Under the one-dimensional Wiener measure $\mathbb P$,  $S^n_k$ is a sum of independent exponential distributed random variables, and $(N^n_t)_{t \in [0,1]}$ is a poisson process with rate $n$ and jump size $\frac{1}{n}$.

For each $n \in \mathbb N$, let us define a generator $\mathcal{L}^n$ via
\begin{equation*}
(\mathcal L_t^n \varphi)(x)
=
n 
\int_{\mathbb{R}}
\left[
\varphi\!\left(x+\frac{1}{\sqrt n} e\right)-\varphi(x)
\right]\eta(de), \quad \forall \, \varphi \in C_b^{2} (\R^d),
\end{equation*}
where $\eta$ satisfies Assumption~\ref{ass:weak_conv} (i). Let $\mathbb P^n \in \mathcal{P}(D([0,\infty);\mathbb R^d)$ be the solution  to the martingale problem for the generator $(\mathcal{L}^n_t)_{t \geq 0}$ with initial distribution $\delta_0$. For each Brownian path $\omega$, define 
\begin{align*}
R_n(\omega): =  ( X_{\cdot} \mapsto (X_{N_t(\omega)})_{t \in [0,1]} )_{\#} \mathbb P^n,
\end{align*}
where $X_{\cdot}$ is the canonical process on $D([0,\infty); \mathbb R^d)$, and thus $\omega \mapsto R_n(\omega) $ is a random variable taking values in $\mathcal M_1(D([0,1],\mathbb R^d))$.

Let us denote by $A$ the set of absolutely continuous, nondecreasing functions $a$ on $[0,1]$ such that $a(0)=0$. We will prove that $R_n$ converges to $\mathbb W$ almost surely, and satisfy a large deviation principle with a rate function given by 
\begin{align}\label{eq:ratefunction}
I(P)
=
\begin{cases}
\displaystyle
\int_0^1 (\dot a(t) \log (\dot a(t))-\dot a(t)+1) \,dt
&
\text{if } P=\operatorname{Law}\big((B_{a(t)})_{t\in[0,1]}\big)
\text{ for some } a \in A,
\\[1.2ex]
+\infty,
&
\text{otherwise}
\end{cases}
\end{align}
where $B$ is a standard $d$-dimensional Brownian motion. Actually, this rate function $I(P)$ is exactly the so-called reciprocal specific relative entropy of $P$ from Wiener measure in dimension $1$; see \cite{BaZh26}. 

\begin{proposition}[LDP for $R_n$]\label{corr:LDP-Rn}

Under the $1$-dimensional Wiener measure $\mathbb P$, the sequence $(R_n)_{n\ge1}$ satisfies a large deviation principle on $\mathcal M_1(D([0,1],\mathbb R^d))$ with speed $n$ and good rate function $I$, i.e. 
\begin{itemize}
\item[{\rm (i)}] For every open set $G \subset \mathcal M_1(D([0,1],\mathbb R^d))$, $$ \liminf_{n \to \infty} \frac{1}{n} \log \mathbb W ( R_n \in G)  \geq - \inf_{P \in G} I(P).$$
\item[{\rm (ii)}] For every closed set $F \subset \mathcal M_1(D([0,1],\mathbb R^d))$, $$ \limsup_{n \to \infty} \frac{1}{n} \log \mathbb W ( R_n \in F)  \leq - \inf_{P \in F} I(P).$$
\item[{\rm (iii)}] For any $k \in \mathbb R_+$, $\{P \in \mathcal M_1(D([0,1],\mathbb R^d)): \, I(P) \leq k\}$ is compact in weak topology. 
\end{itemize}
\end{proposition}

\subsection{SEMOT strong duality}
\label{sec:duality}
% Following the approach introduced in the seminal work of Touzi and Tan (AoP, 2013), establishing duality for the semimartingale optimal transport problem requires showing that the set of admissible transport plans is weakly compact. This is typically achieved when the cost function grows superlinearly with respect to the volatility coefficient.
% As discussed above, there are several possible constructions for approximating the limiting martingale. In this section, we adopt as our cost function the superlinear KL limit associated with the second scheme from Section \ref{sec:limit_of_chaos}.\\

We now explain how the previous KL asymptotics can be used to define a meaningful cost functional for MOT. The main idea is to select a Poissonization scheme for which the limiting entropy cost is superlinear with respect to the volatility coefficient, as this growth is the key tightness/compactness mechanism behind the strong duality argument.
For this reason, throughout this subsection we work with the second Poissonization scheme from Remark~\ref{rmk:twoschemes}. 

% Now, we can choose a cost in the form of generalized specific relative entropy for which MOT is well defined. Then we can prove that there is duality. In this goal, we will chose a particular Poissonization scheme: the second scheme from Section \ref{sec:limit_of_chaos}, where the KL limit is superlinear wrt to the volatility coefficient.\\

Let \(\xi:C([0,1];\mathbb R^d)\to\mathbb R\)  be a bounded Borel-measurable cost functional, continuous with respect to the uniform topology. Fix a reference volatility decomposition as in \eqref{eq:ref_volatility}.
For \(\Sigma_1\in\mathbb S_+^d\),  define $\lambda_1$ and $\bar\Sigma_1$ as in trace normalization scheme \eqref{eq:target_measure} 
and the running cost $\ell^{\mathrm{tr}}:[0,1]\times \mathbb R^d \times \mathbb{S}_d^{+}\to\mathbb R$,
\[
\ell^{\mathrm{tr}}(t,x,\Sigma_1):=
\begin{cases}
\lambda_1\log\frac{\lambda_1}{\lambda_2(t,x)}-\lambda_1+\lambda_2(t,x)-\dfrac{\lambda_1}{2}\log\det\!\left(\bar{\Sigma}_2^{-1}(t,x)\bar{\Sigma}_1\right),
& \Sigma_1\in\mathbb S_{++}^d,\\[4pt]
\lambda_2(t,x), & \Sigma_1=0,\\[4pt]
+\infty, & \Sigma_1\in\partial\mathbb S_+^d\setminus\{0\}.
\end{cases}
\]
\noindent
For \(\mathbb P\in\mathcal P\), set
$
J(\mathbb P):=
\mathbb E^{\mathbb P}\!\left[\int_0^1 \ell^{\mathrm{tr}}(t,X_t,\Sigma_{1,t}^{\mathbb P})\,dt + \xi(X_{1\wedge\cdot})\right].
$
The primal value function is then
\begin{equation}
\label{eq:primal}
V(\mu_0,\mu_1):=
\inf_{\mathbb P\in\mathcal P(\mu_0,\mu_1)} J(\mathbb P).
\end{equation}
and its dual formulation is
\begin{equation}
\mathcal{V}(\mu_0,\mu_1) := \sup_{\phi_1 \in C_b(\mathbb{R}^d)}
\left\{ \mu_0(\phi_0) - \mu_1(\phi_1) \right\},
\label{eq:dual}
\end{equation}
where
\[
\phi_0(x) := \inf_{\substack{\mathbb{P} \in \mathcal{P} \\ X_0 = x \ \mathbb{P}\text{-a.s.}}}
\mathbb{E}^{\mathbb{P}}\!\left[
\int_0^1 \ell^{\mathrm{tr}}(t,X_t,\Sigma_{1,t}^{\mathbb P})\,dt + \phi_1(X_1)
 + \xi(X_{1\wedge\cdot})\right].
\]

\begin{remark}
Here, the cost $\xi$ can encode a general financial payoff. For example, one can take an Asian option payoff such that
\[
\xi(X_{T\wedge\cdot})
=
\left(\frac{1}{T}\int_0^T X_t\,dt-K\right)_+ ,
\]
where $K>0$ denotes the strike price.
\end{remark}

\begin{assumption}\label{ass:ref_process_duality}
There exist constants \(0<\underline b\le \overline b<\infty\) and \(M>0\) such that
\[
\underline b\le \lambda_2(t,x)\le \overline b,
\qquad
\bar\Sigma_2(t,x)\preceq M I_d,
\qquad (t,x)\in[0,1)\times\mathbb R^d.
\]
We also assume that $\bar \Sigma_2$ an $\lambda_2$ are continuous.
\end{assumption}

\begin{theorem}[Strong Duality]
\label{theorem:strong duality}
Let $\mu_0,\mu_1\in\mathcal{M}_1(\mathbb{R}^d)$. Assume $\mathcal{P}(\mu_0,\mu_1) \not = \emptyset $ and $V(\mu_0,\mu_1)<\infty$.  Then under Assumption \ref{ass:ref_process_duality}, duality holds
\[
V(\mu_0,\mu_1)=\mathcal{V}(\mu_0,\mu_1),
\]
and the primal infimum is attained. 
\end{theorem}

The proof follows the standard convex-analytic strategy for transport duality. Fixing $\mu_0$, we first show that the map
$
\mu_1 \mapsto V(\mu_0,\mu_1)
$
is convex and lower semicontinuous, which allows us to apply a Fenchel--Moreau type duality theorem on the space of signed measures. Lower semicontinuity comes from a tightness and compactness argument, while convexity is obtained by mixing admissible laws. Once these two properties are established, the convex conjugate is identified through the auxiliary control problem defining $\phi_0$, which yields the desired strong duality formula.

\section{Numerical experiments}

In this section, we consider a special case \(\xi \equiv 0\) in the optimization problem \eqref{eq:primal}. The reference volatility is Brownian, that is, $\lambda_2 \equiv 1, ~ \bar\Sigma_2\equiv I_d$.  We first derive the coupled HJB--FP system, then present a Sinkhorn-type algorithm and its numerical simulations in the one- and two-dimensional settings. Finally, we formally prove that the dual functional is monotone increasing, which ensures the convergence of the algorithm.

\subsection{Dynamic programming and HJB characterization}
\label{sec:HJB}
Due to the strong duality established in Theorem~\ref{theorem:strong duality}, we have the following dynamic value function:

\[
\phi(t,x) := \inf_{\mathbb{P} \in \mathcal{P}} 
\mathbb{E}^{\mathbb{P}} \left[
\int_t^1 \ell^{\mathrm{tr}}(\Sigma_{1,s}^{\mathbb P})\, ds 
+ \phi_1(X_1)\Big| ~X_t=x\right].
\]
\noindent
Assuming $\phi$ is locally bounded and the Hamiltonian continuous, $\phi$ is a viscosity solution of the dynamic programming equation
\[
\partial_t \phi(t,x)
 +
\inf_{\Sigma_1\in\mathbb{S}_d^{+}}
\left\{
\ell^{\mathrm{tr}}(\Sigma_1)
+
\frac{1}{2}\operatorname{Tr}
\big(\Sigma_1 D_x^2 \phi\big)
\right\}
= 0,
\qquad (t,x)\in [0,1)\times \mathbb{R}^d.
\]

\noindent
with terminal condition
$
\phi(1,x) = \phi_1(x).
$

\begin{remark}[Local boundedness of $\phi$]
Consider the constant control
$
\Sigma_1(s) \equiv I_d$ for $s\in[t,1].$
Given $\ell^{\mathrm{tr}}(I_d)=0$, evaluating the cost functional at this admissible control yields
$
\phi(t,x)
\le
\mathbb E\!\left[
\phi_1(X_1)
\right].
$
Since $\phi_1$ is bounded, then
$
\phi(t,x)
\le
\|\phi_1\|_\infty.
$
Moreover, $\ell^{\mathrm{tr}}$ nonnegative yields $\phi(t,x)\ge -\|\phi_1\|_\infty$.
Hence $\phi$ is locally bounded on $[0,1]\times\mathbb R^d$.
\end{remark}

\begin{remark}[Continuity of the Hamiltonian]
\label{rem:continuity_hamiltonian}
The function
\[
H(\Gamma):=\inf_{\Sigma_1\in\mathbb S^d_+}
\left\{\frac12\tr(\Sigma_1\Gamma)+\ell^{\mathrm{tr}}(\Sigma_1)\right\},
\qquad \Gamma\in\mathbb S^d,
\]
is concave, being the infimum of affine functions of $\Gamma$. To prove continuity, it is enough to show that $H$ is finite everywhere.
Indeed, since $\bar\Sigma_2 = I_d$,
we have $\tr(\bar\Sigma_1)=d$, hence $\det(\bar\Sigma_1)\le 1$ by the arithmetic--geometric mean inequality, and thus 
$
-\frac{\lambda_1}{2}\log\det(\bar\Sigma_1)\ge 0.
$
Moreover, we also get
$
\tr(\Sigma_1\Gamma)\ge \lambda_1 d\,\lambda_{\min}(\Gamma),
$ where $\lambda_{\min}(\Gamma)$ denotes the smallest eigenvalue of $\Gamma$.
Therefore
\[
\frac12\tr(\Sigma_1\Gamma)+\ell^{\mathrm{tr}}(\Sigma_1)
\ge
\lambda_1\log\lambda_1-\lambda_1+1
+\frac{\lambda_1 d}{2}\lambda_{\min}(M),
\]
and the right-hand side is bounded from below in $\lambda_1\ge0$ because
$\lambda_1\log\lambda_1$ is superlinear. Since also $\ell^{\mathrm{tr}}(0)=1<\infty$, we obtain $H(\Gamma)\in\mathbb R$ for every $\Gamma$, thus $\operatorname{dom}(H)=\mathbb S^d$.
\end{remark}

\begin{remark}
Although the admissible set is $\mathbb S^d_{+}$, every minimizer belongs to $\mathbb S^d_{++}$. On $\partial\mathbb S^d_{+}\setminus\{0\}$, the objective is infinite  and $\Sigma_1=0$ cannot be optimal. Indeed, for any $\Gamma\in\mathbb S^d$, since $\Sigma_1=\lambda_1\,\bar\Sigma_1$ gives
$
\ell^{\mathrm{tr}}(\Sigma_1)+\frac12\tr(\Sigma_1\Gamma)
=
\lambda_1\log\lambda_1-\lambda_1+1+\frac{\lambda_1}{2}\tr(\bar\Sigma_1\Gamma),
$
whose minimum over $\lambda_1>0$ is attained at
$
\lambda_1=\exp\!\Big(-\tfrac12\tr(\bar\Sigma_1\Gamma)\Big)>0,
$
with value
$1-\exp\!\big(-\tfrac12\tr(\bar\Sigma_1\Gamma))<1=\ell^{\mathrm{tr}}(0).
$\newline
\end{remark}

We now compute the infimum explicitly.
For each $(t,x)\in[0,1)\times\mathbb R^d$, set
$
\Gamma(t,x):=D_x^2\phi(t,x).
$
Then
\[
H\bigl(D_x^2 \phi(t,x)\bigr)
=
1-\exp\!\Big(
-\tfrac12\Big[d(1-\mu(t,x))+\log\det\big(\Gamma(t,x)+\mu(t,x)I_d\big)\Big]
\Big),
\]
where, for each fixed $(t,x)$, the quantity $\mu(t,x)\in\mathbb R$ is the unique scalar, Lagrange multiplier, such that
\[
\tr\Big(\Gamma(t,x)+\mu(t,x)I_d\Big)^{-1}=d,
\qquad
\Gamma(t,x)+\mu(t,x)I_d\succ0.
\]
\noindent
The optimal diffusion matrix is then given by
\[
\Sigma_1^*(t,x)
=
\left(1-H\bigl(D_x^2 \phi(t,x)\bigr) \right)\,
\big(\Gamma(t,x)+\mu(t,x)I_d\big)^{-1}.
\]
\noindent
The Fokker--Planck equation characterizing the marginal distributions of the
optimal martingale reads
\[
\partial_t p
= \frac{1}{2}\,\nabla_x^2 : \bigl(\Sigma_1^{*} p\bigr),
\qquad
\nabla_x^2 : \bigl(\Sigma_1^{*} p\bigr)
= \sum_{i,j=1}^d
\partial_{x_i x_j}\Bigl(\Sigma_{1,ij}^{*}(t,x)\,p(t,x)\Bigr).
\]

\subsection{1D Case}

In this section, we present a Sinkhorn-type algorithm for the numerical solution of the coupled HJB--Fokker--Planck system introduced above, in the one-dimensional periodic setting. The goal is to compute simultaneously the optimal volatility surface and the associated density $p(t,x)$ arising from the stochastic control problem.

The key point is that the coupled HJB--Fokker--Planck system can be interpreted as a fixed-point problem on the terminal potential. Starting from a guess $\phi_1$, one first solves the backward HJB equation to recover the value function $\phi$. This determines the optimal diffusion coefficient $\Sigma_1^*$, or equivalently $\sigma$ in dimension one. Keeping this diffusion fixed, one then solves the forward Fokker--Planck equation and obtains the density $p$, in particular its terminal value $p_1:=p(T,\cdot)$. The terminal potential is then updated by comparing $p_1$ with the prescribed target marginal $\mu_1$ through a Sinkhorn-type correction:
\[\phi_1^{\rm new}:=\phi_1+ \eta \log(p_1/\mu_1).\]
 In this way, the numerical procedure can be viewed as the search for a fixed point of the mapping
\[
\phi_1 \mapsto \phi \mapsto \sigma \mapsto p \mapsto p_1 \mapsto \phi_1^{\rm new},
\]
where the last arrow is the Sinkhorn-type update. At convergence, the induced terminal density satisfies $p_1=\mu_1$, so that both the HJB--Fokker--Planck system and the terminal marginal constraint are satisfied simultaneously.

We now describe the discretization of this procedure. The problem is posed on the time--space cylinder $[0,T]\times\mathbb T$, where $\mathbb T=\mathbb R/\mathbb Z$ denotes the one-dimensional torus, and we use the uniform grids
$
t_n=n\Delta t$, for $n=0,\dots,N_t,
$
and
$
x_i=i\Delta x$, for $i=0,\dots,N_x-1,
$
with
$
\Delta t=\frac{T}{N_t}
$
and
$
\Delta x=\frac{1}{N_x}.
$
The discrete unknowns are
\[
\phi_i^n \approx \phi(t_n,x_i),
\qquad
p_i^n \approx p(t_n,x_i),
\qquad
\sigma_i^n \approx \sigma(t_n,x_i).
\]
\noindent
Since the spatial domain is periodic, periodic boundary conditions are imposed throughout the numerical solution of the PDE system. The backward HJB equation is solved implicitly in time, and the nonlinearity induced by the term $\exp\!\left(-\tfrac12\Delta\phi\right)$ is treated by Newton's method, using a sparse periodic finite-difference Jacobian. Once $\phi$ has been computed, the forward Fokker--Planck equation is also solved implicitly in time. At each step, this yields a linear periodic tridiagonal system, assembled from the diffusion coefficient $\sigma$ obtained in the backward pass and solved directly with `spsolve`.

Before performing the Sinkhorn update, we smooth the log-ratio in order to improve numerical stability. More precisely, if
$
r=\log\!\left(\frac{p_1}{\mu_1}\right),
$
then at each smoothing iteration we replace $r$ by its local weighted average
$
r_i \leftarrow \frac14\,r_{i-1}+\frac12\,r_i+\frac14\,r_{i+1}.
$
We also impose a lower bound of $10^{-1}$ on the densities in order to avoid numerical instabilities when the computed density becomes too small.

The full one-dimensional algorithm is summarized in the pseudocode below.
The parameter $\eta$ plays the role of a relaxation parameter in the Sinkhorn update and is crucial for stability. From the numerical viewpoint, it controls the size of the correction applied to the terminal potential at each outer iteration. From the theoretical viewpoint, since the dual correction is linked to the time-discretized entropy term, it is natural to choose $\eta$ of the same order as $\Delta t$. In all our experiments, we therefore take $\eta$ close to $\Delta t$.

In the 1D test case (Figure~\ref{fig1:results}), the initial marginal \(\mu_0\) is chosen as a periodized Gaussian density on the torus, centered at \(x=0.5\) with standard deviation \(0.05\). The terminal marginal \(\mu_1\) is defined as the mixture
\[
\mu_1(x)
=
q\,\mathcal{N}_{\mathbb{T}}(x;0.5,s_0)
+\frac{1-q}{2}\,\mathcal{N}_{\mathbb{T}}(x;0.5-d_1,s_1)
+\frac{1-q}{2}\,\mathcal{N}_{\mathbb{T}}(x;0.5+d_1,s_1),
\]
with \(q=0.6\), \(d_1=0.2\), \(s_0=0.1\), and \(s_1=0.05\), where \(\mathcal{N}_{\mathbb{T}}\) denotes the periodized Gaussian density on the torus. Both \(\mu_0\) and \(\mu_1\) are normalized to have unit mass. The spatial domain is \([0,1]\). We set \(T=0.1\), \(N_x=128\), and \(N_t=80\), which gives
$
\Delta x=0.008$ and
$\Delta t=0.001$. Finally, $\eta=0.001$.
We observe strong convergence after roughly 200 iterations. The $L^1$ error plot suggests an exponential rate of convergence.

\begin{remark}
In this section, we work with $T=0.1$ as on longer time horizon the diffusion already brings the terminal density very close to the target after one iteration, making the test case uninformative.
\end{remark}

\subsection{2D case}
In the two-dimensional setting, the same fixed-point strategy is used, but additional stabilization is needed in order to maintain a robust convergence of the outer Sinkhorn iterations. In particular, we use an adaptive choice of the parameter $\eta$: the learning rate is decreased whenever the $L^1$ error increases from one iteration to the next, and increased slightly otherwise. In addition, we incorporate Anderson acceleration, which replaces the standard update by a weighted combination of the last $m$ candidate iterates. The weights are computed by solving a small regularized linear system based on the residual history, so as to reduce the residual in the span of previous updates. These modifications are important in dimension two, where the iterations are noticeably slower and each PDE solve is more costly than in the one-dimensional case. 
Here (Figure ~\ref{fig2:results}), the initial and terminal marginals are chosen as periodized anisotropic Gaussian densities on the torus $\mathbb T^2$. The initial density $\mu_0$ is centered at $(0.5,0.5)$ and aligned with the coordinate axes, with standard deviations $0.03$ and $0.07$ in the two spatial directions. The target density $\mu_1$ is centered at the same point, with standard deviations $0.08$ and $0.12$, and is rotated by an angle $\theta=\pi/6$. Both marginals are periodized on $[0,1]^2$ and normalized so as to define probability densities.
For the discretization, we take $T=0.1$, $N_x=64$, and $N_t=40$, which gives
$
\Delta x=0.016
$
and
$
\Delta t=0.0025.
$
The adaptive parameter $\eta$ is initialized at $0.005$.

With these choices, we again observe convergence of the algorithm, here after roughly 100 outer iterations.

\begin{algorithm}[H]
\caption{Sinkhorn Iterations for Periodic HJB--Fokker--Planck System in 1D}
\begin{algorithmic}[1]
\Require Initial density $\mu_0$, target density $\mu_1$, time step $\Delta t$, space step $\Delta x$, number of space points $N_x$, number of time steps $N_t$
\Ensure Approximate terminal potential $\phi_1$

\State Initialize $\phi_1 \gets 0$
\For{$k = 1$ to $N_{\mathrm{outer}}$}

    \State \textbf{Backward step: solve HJB equation with terminal condition
    $
    \phi(1,x)=\phi_1(x)
    $
}
    \State Compute $\{\phi^n\}_{n=0}^{N_t}$ by implicit time stepping:
    $
    \phi^{n+1} - \phi^n = \Delta t \left(1 - \exp\!\left(-\tfrac12 \Delta \phi^{n+1}\right)\right)
    $
    \State \textbf{Compute diffusion coefficients}
    Compute $\{\phi^n\}_{n=0}^{N_t}$
   by $\sigma^n  = \exp\!\left(-\tfrac12 \Delta \phi^n\right)$
    \State \textbf{Forward step: solve Fokker--Planck equation}
    \State Starting from $p^0=\mu_0$, compute $\{p^n\}_{n=0}^{N_t}$ by
    $
    p^{n+1} - p^n = \frac{\Delta t}{2}\Delta (\sigma^{n+1} p^{n+1})
    $
    \State \textbf{Evaluate mismatch at final time}
    \State $p^1 \gets p^{N_t}$
    \State Compute error
    $
    E_k \gets \|p^1 - \mu_1\|_{L^1}
    $
    \If{$E_k < \varepsilon$}
        \State \Return $\phi_1$
    \EndIf

    \State \textbf{Update terminal potential}
    \State Compute
    $
    r \gets \log\!\left(\frac{p^1}{\mu_1}\right)
    $
    \State Replace $r$ by its periodic smoothing
    \State Update
    $
    \phi_1 \gets \phi_1 + \eta\, r
    $
\EndFor
\State \Return $p^1,\{\sigma^n\}_{n=0}^{N_t}$
\end{algorithmic}
\end{algorithm}

\begin{figure}[htbp]
    \centering

    % First row
    \begin{subfigure}[t]{0.48\textwidth}
        \centering
        \includegraphics[width=\linewidth]{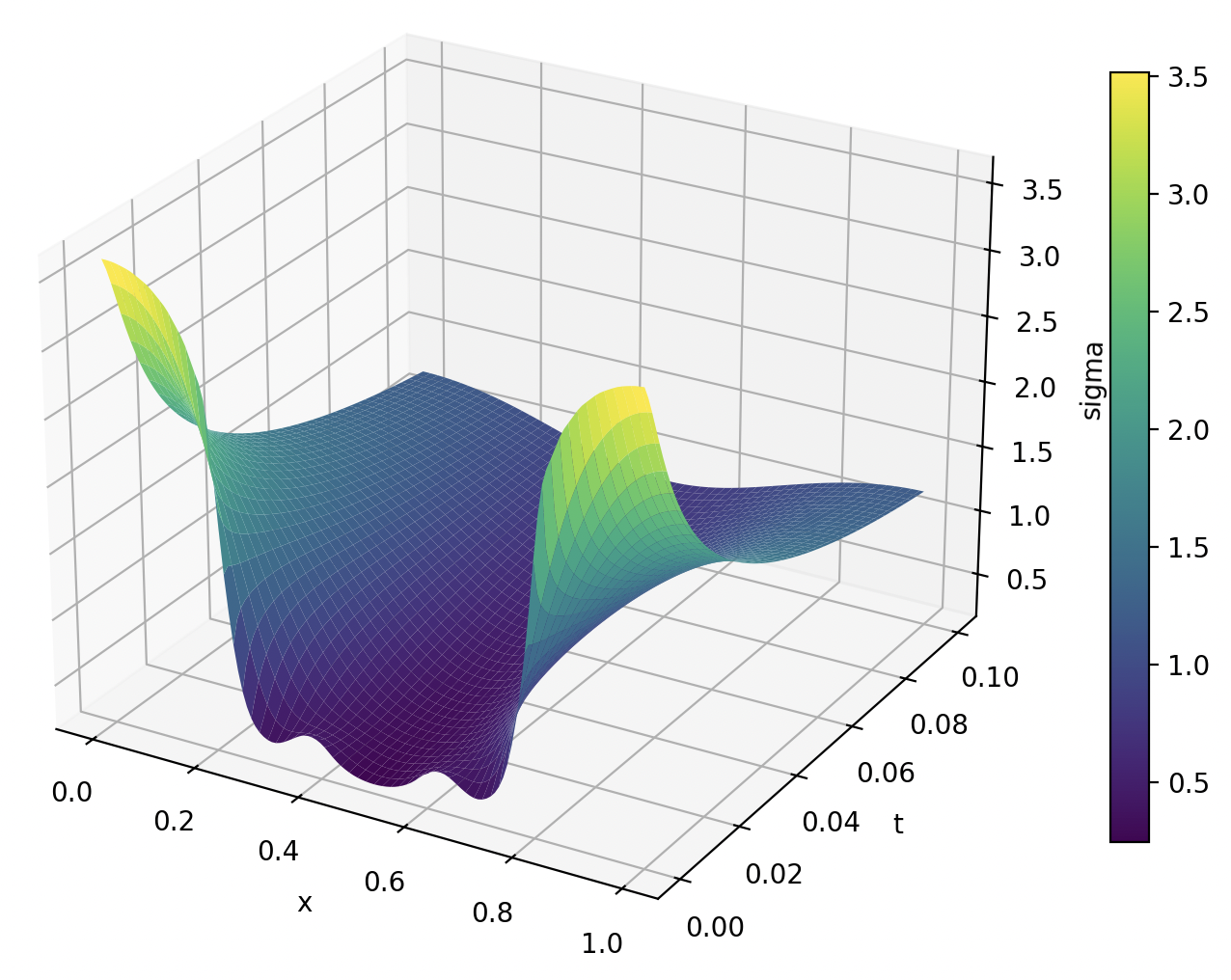}
        \caption{Final volatility surface}
        \label{fig1:a}
    \end{subfigure}
    \hfill
    \begin{subfigure}[t]{0.48\textwidth}
        \centering
        \includegraphics[width=\linewidth]{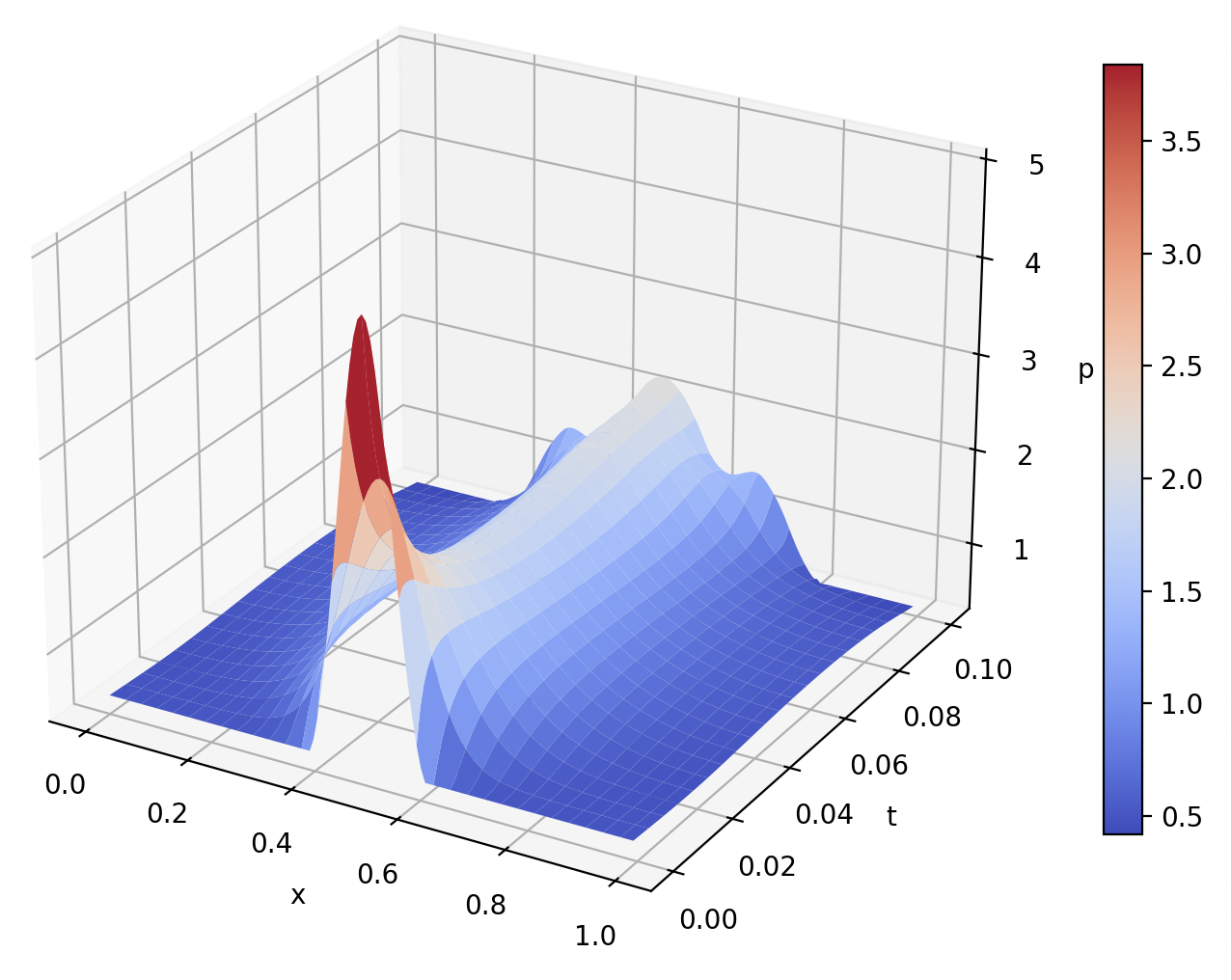}
        \caption{Marginal surface $p(t,x)$}
        \label{fig1:b}
    \end{subfigure}

    \vspace{0.5cm}

    % Second row
    \begin{subfigure}[t]{0.48\textwidth}
        \centering
        \includegraphics[width=\linewidth]{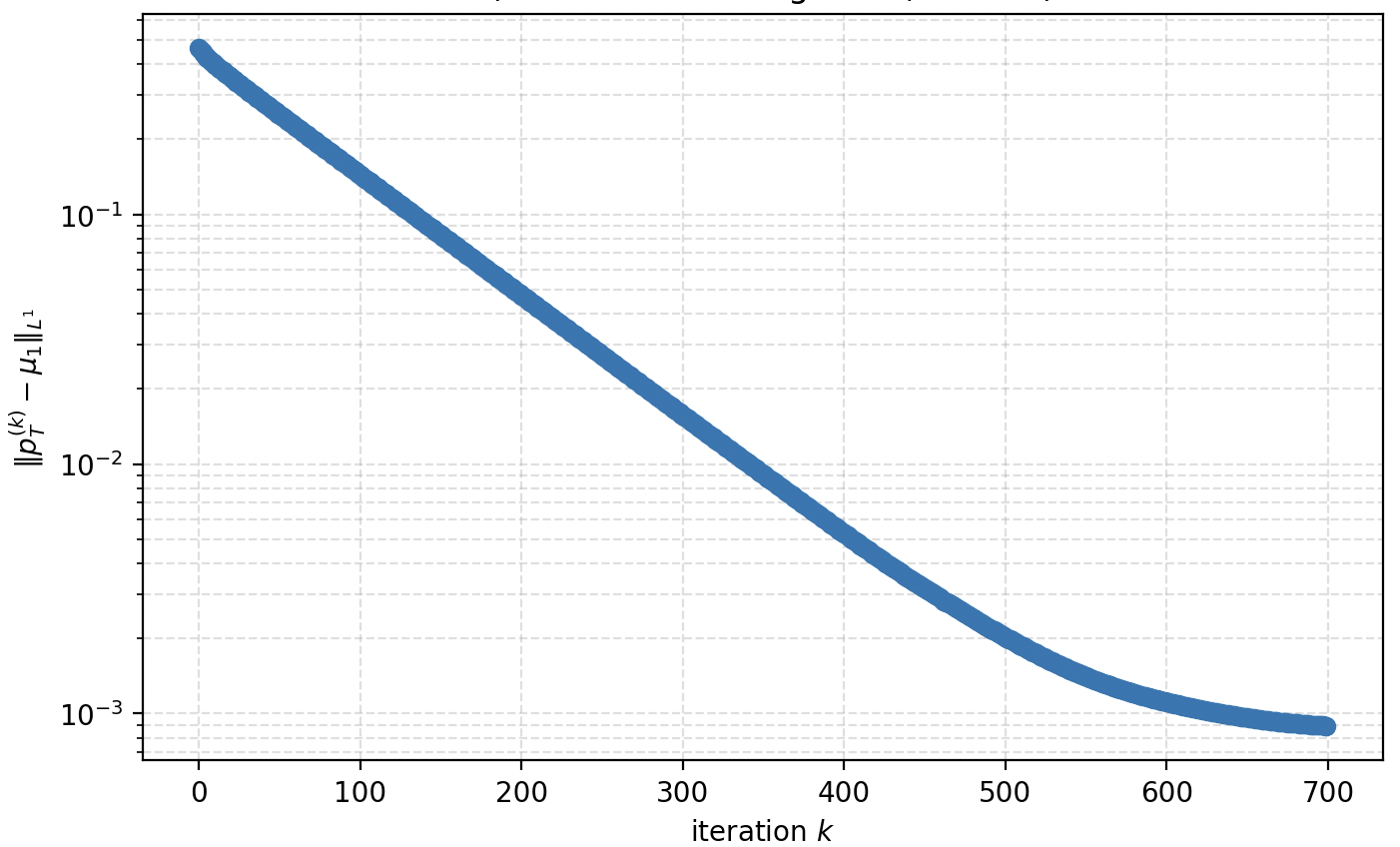}
        \caption{Sinkhorn convergence of the $L^1$ error}
        \label{fig1:c}
    \end{subfigure}
    \hfill
    \begin{subfigure}[t]{0.48\textwidth}
        \centering
        \includegraphics[width=\linewidth]{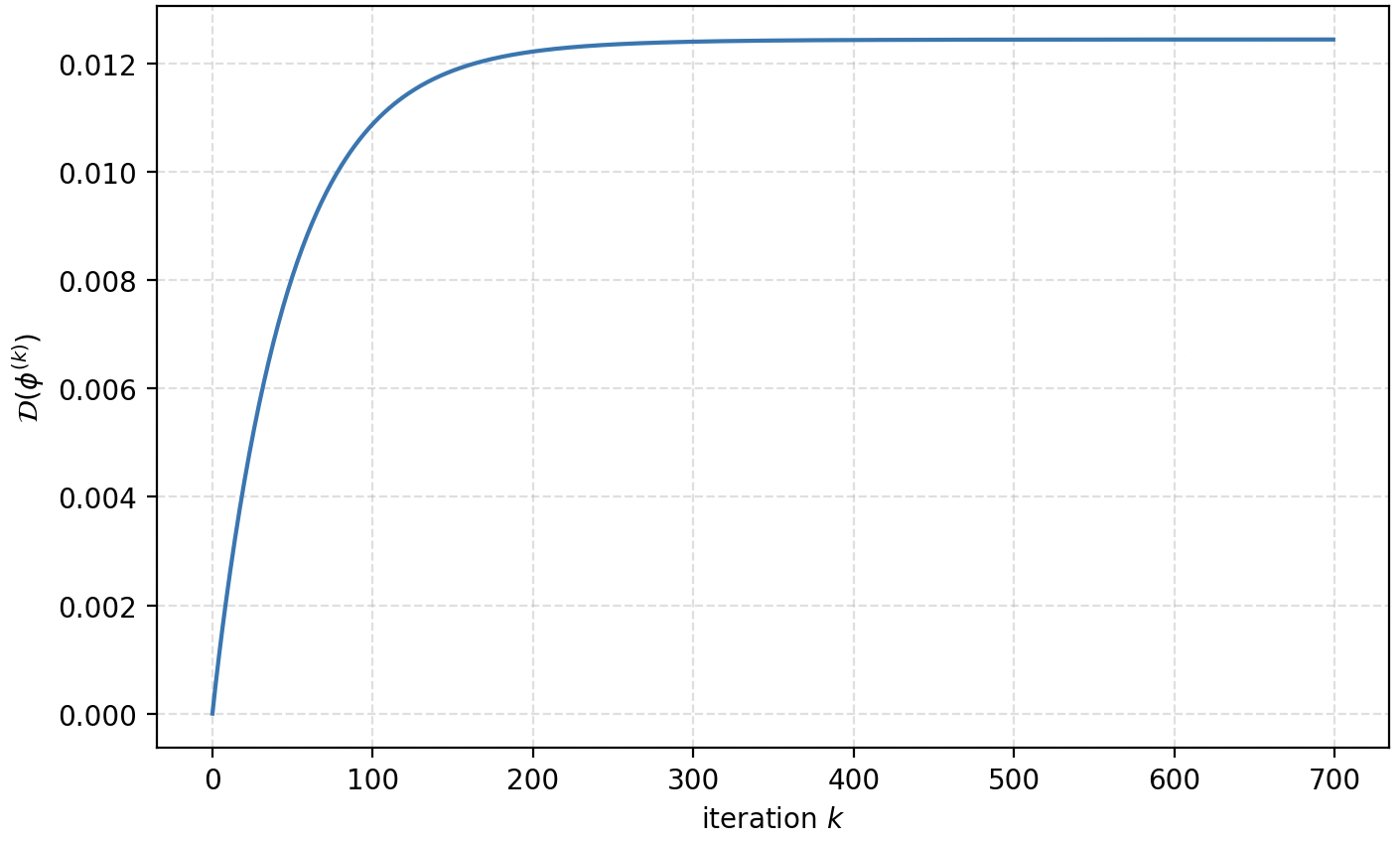}
        \caption{Monotone dual functional $\mathcal{D}(\phi^{(k)})$}
        \label{fig1:d}
    \end{subfigure}

    \vspace{0.5cm}

    % Third row centered
    \begin{subfigure}[t]{0.48\textwidth}
        \centering
        \includegraphics[width=\linewidth]{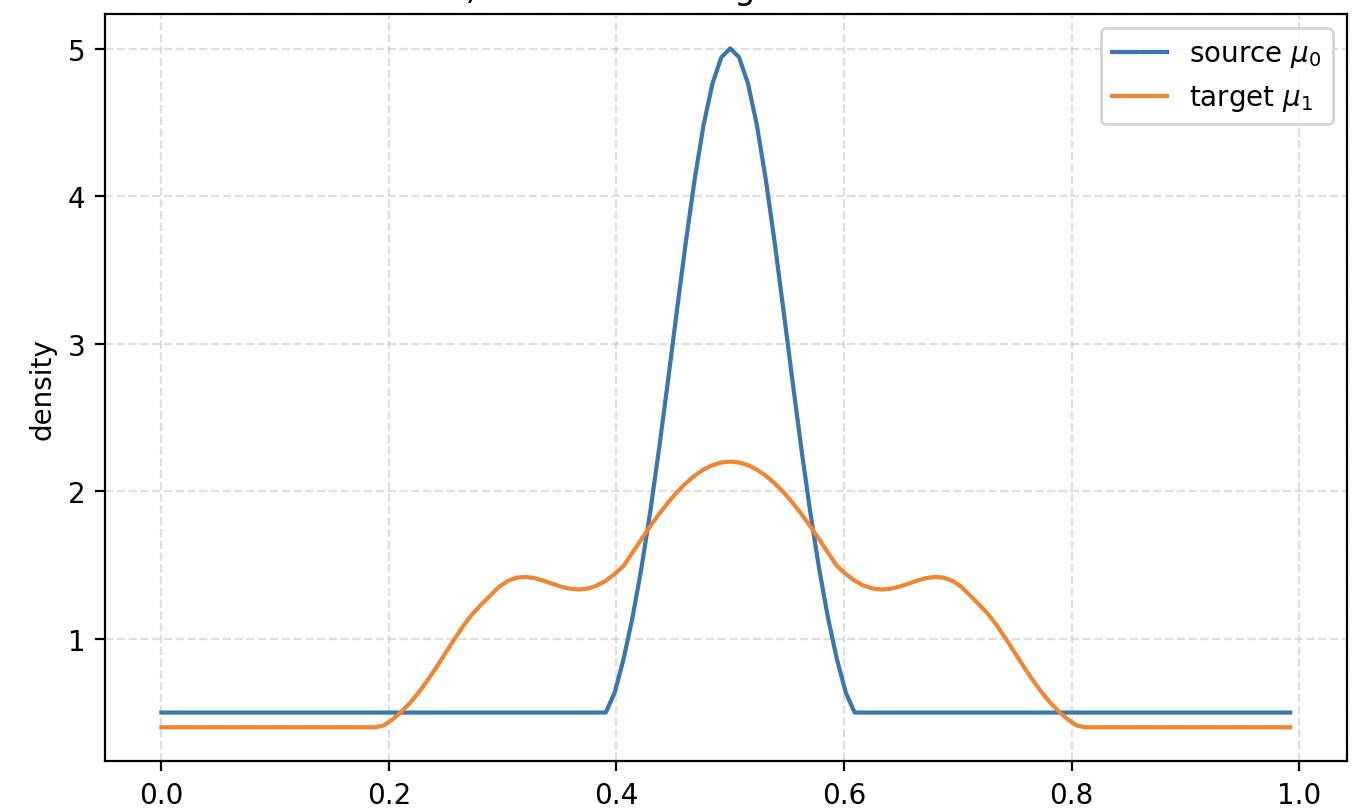}
        \caption{Source and target distributions}
        \label{fig1:e}
    \end{subfigure}
    
    \caption{Gaussian to Gaussian mixture with Brownian reference volatility}
    \label{fig1:results}
\end{figure}

\begin{figure}[htbp]
    \centering

    % First row
    \begin{subfigure}[t]{0.48\textwidth}
        \centering
        \includegraphics[width=\linewidth]{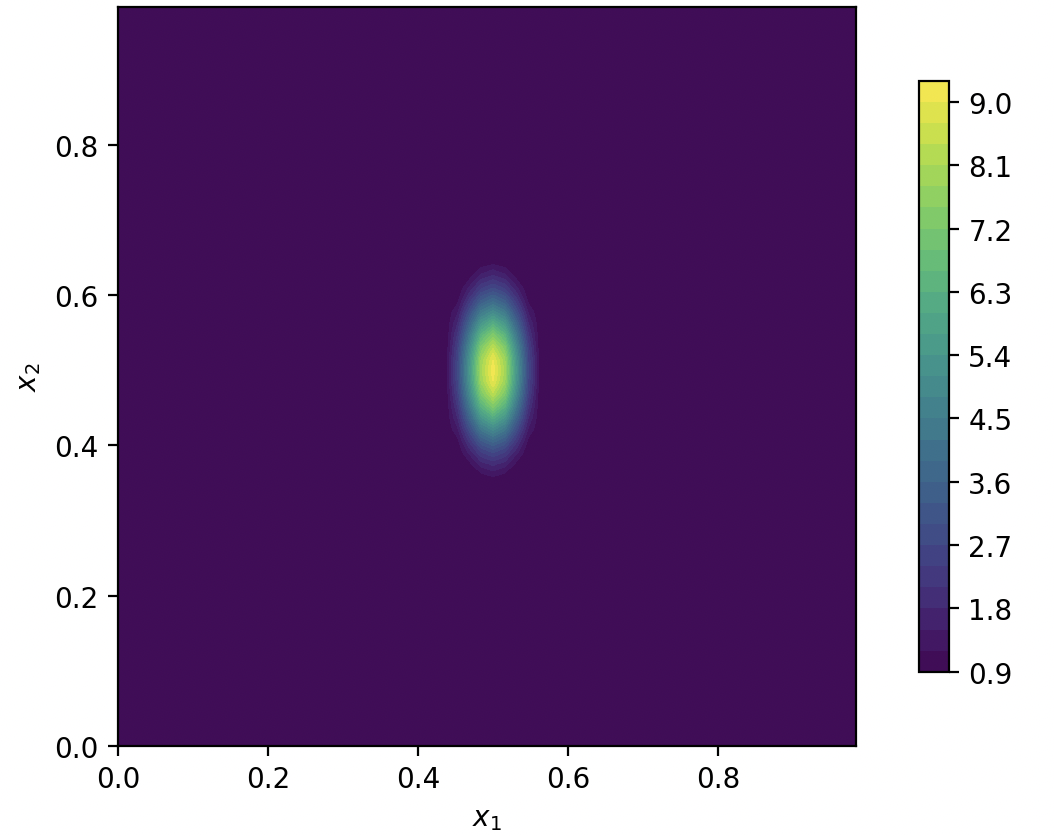}
        \caption{Source marginal}
        \label{fig2:a}
    \end{subfigure}
    \hfill
    \begin{subfigure}[t]{0.48\textwidth}
        \centering
        \includegraphics[width=\linewidth]{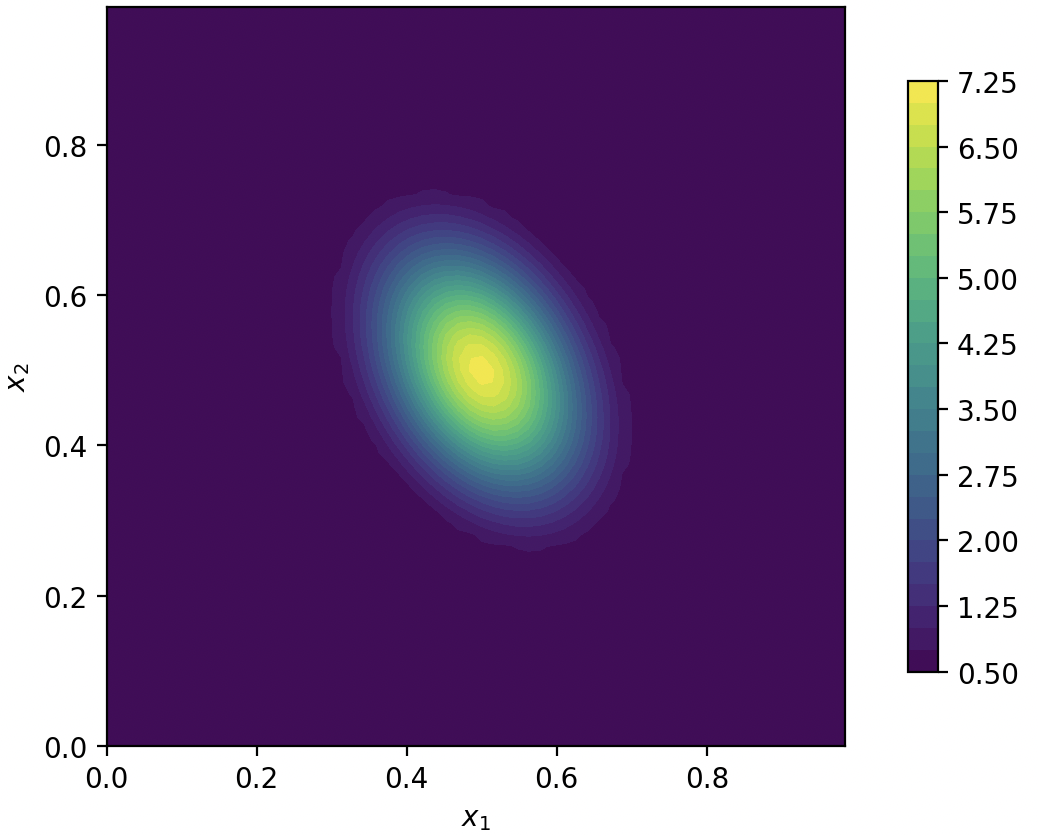}
        \caption{Target marginal}
        \label{fig2:b}
    \end{subfigure}

    \vspace{0.5cm}

    % Second row
    \begin{subfigure}[t]{0.48\textwidth}
        \centering
        \includegraphics[width=\linewidth]{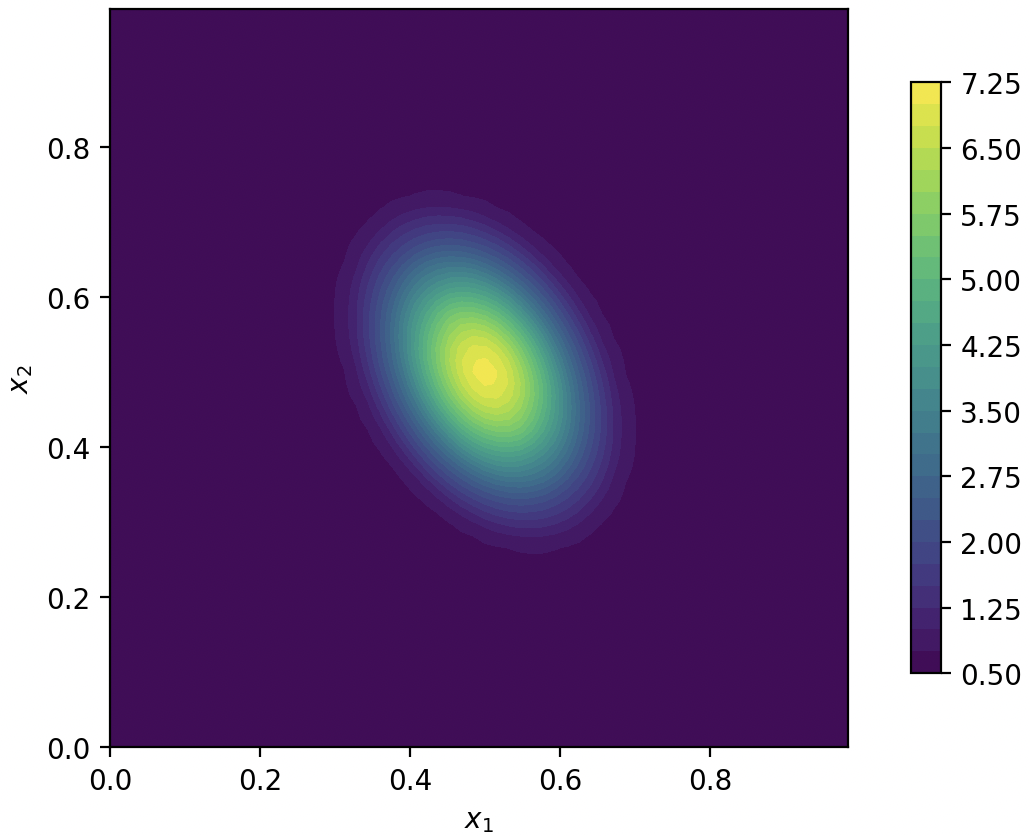}
        \caption{Final marginal $p_1$}
        \label{fig2:c}
    \end{subfigure}
    \hfill
    \begin{subfigure}[t]{0.48\textwidth}
        \centering
        \includegraphics[width=\linewidth]{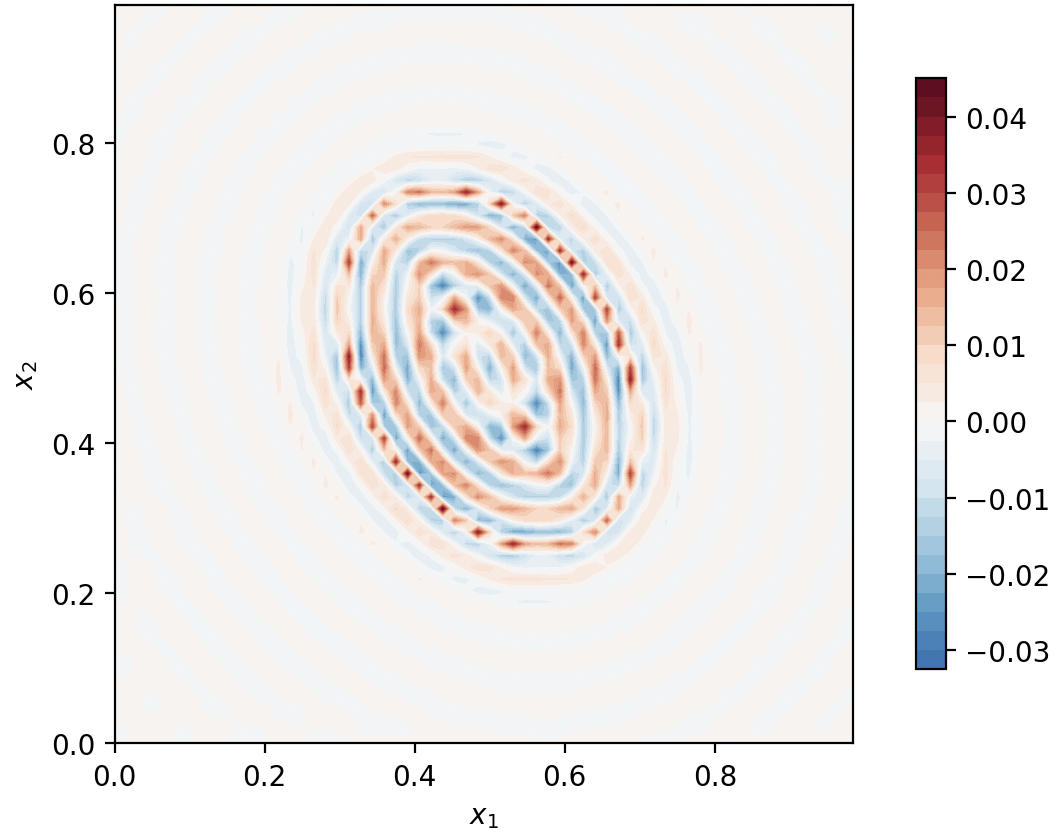}
        \caption{$p_1 - \mu_1$}
        \label{fig2:d}
    \end{subfigure}

    \vspace{0.5cm}

    % Third row
    \begin{subfigure}[t]{0.48\textwidth}
        \centering
        \includegraphics[width=\linewidth]{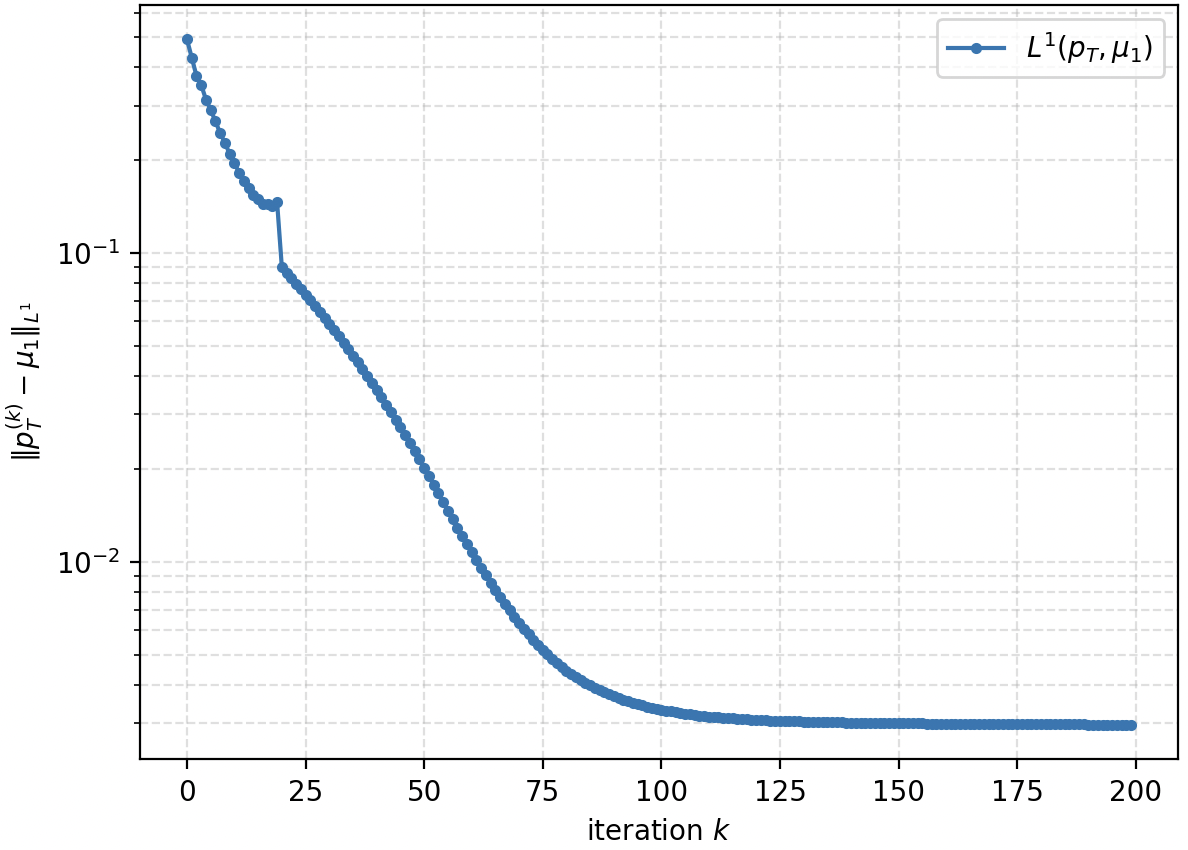}
        \caption{Sinkhorn convergence of the $L^1$ error}
        \label{fig2:e}
    \end{subfigure}
    \hfill
    \begin{subfigure}[t]{0.48\textwidth}
        \centering
        \includegraphics[width=\linewidth]{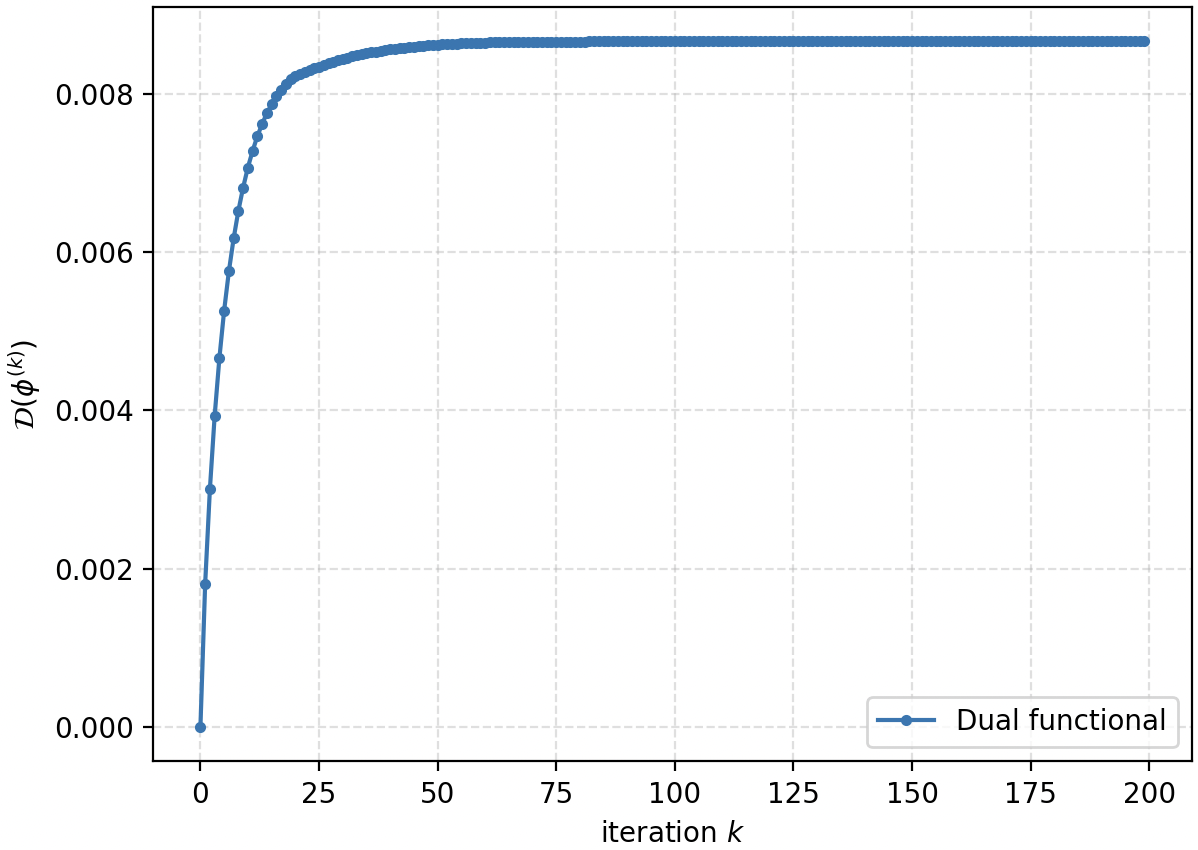}
        \caption{Monotone dual functionnal $\mathcal{D}(\phi^{(k)})$}
        \label{fig2:f}
    \end{subfigure}

    \caption{Gaussian to Gaussian in 2D with Brownian reference volatility}
    \label{fig2:results}
\end{figure}

\subsection{Monotonicity of the dual functional}

Let $H$ be defined as in Remark \ref{rem:continuity_hamiltonian}
and, for a terminal condition $\psi:\R^d\to\R$, let $(\phi^\psi,p^\psi)$ be a smooth solution of
\begin{equation}\label{eq:HJB-FP}
\left\{
\begin{aligned}
&\partial_t \phi^\psi + H(\nabla_x^2 \phi^\psi)=0,
&& (t,x)\in [0,1)\times \R^d,\\
&\partial_t p^\psi = \frac12 \nabla_x^2 : \bigl(\Sigma^\psi p^\psi\bigr),
&& (t,x)\in (0,1]\times \R^d,\\
&\phi^\psi(1,\cdot)=\psi,\qquad p^\psi(0,\cdot)=\mu_0,
\end{aligned}
\right.
\end{equation}
where
$
\Sigma^\psi(t,x):=\Sigma^*(\nabla_x^2 \phi^\psi(t,x))
$
is the minimizer in the definition of $H$.
Let $\mu_1$ be the target terminal law, and define the dual functional
\[
\mathcal D(\psi):=\int_{\R^d} \phi^\psi(0,x)\,\mu_0(dx)-\int_{\R^d}\psi(x)\,\mu_1(dx).
\]
Assume moreover that $p_1^\psi$ and $\mu_1$ admit strictly positive smooth densities, still denoted by $p_1^\psi$ and $\mu_1$.
Consider the flow
\begin{equation}\label{eq:Sinkhorn-flow}
\partial_s \psi_s = \log\!\left(\frac{p_1^{\psi_s}}{\mu_1}\right).
\end{equation}
Then, formally,
\begin{equation}\label{eq:frac_D}
\frac{d}{ds}\mathcal D(\psi_s)
=
\int_{\R^d}
\bigl(p_1^{\psi_s}-\mu_1\bigr)
\log\!\left(\frac{p_1^{\psi_s}}{\mu_1}\right)\,dx
= 
\KL(p_1^{\psi_s}\mid \mu_1)
+
\KL(\mu_1\mid p_1^{\psi_s})
\ge 0.
\end{equation}

Hence $\mathcal D(\psi_s)$ is monotonically increasing along the flow \eqref{eq:Sinkhorn-flow}; see Appendix \ref{sec:monotonicity} for the formal computation. In this sense, \eqref{eq:Sinkhorn-flow} can be viewed as a mirror-ascent-type evolution for the dual problem, with $\mathcal D$ playing the role of a Lyapunov functional. Indeed, the right-hand side of (\ref{eq:frac_D}) vanishes if and only if $p_1^{\psi_s}=\mu_1$. Thus the only stationary points of the flow are precisely the terminal potentials for which the terminal marginal constraint is satisfied.

This Lyapunov structure provides a formal explanation for the convergence of the Sinkhorn-type iteration: as the flow evolves, the dual functional increases while the terminal mismatch is driven to zero. This monotone behavior is clearly visible in plot (d) of Figures~\ref{fig1:results} as well as in plot (f) of Figures~\ref{fig2:results}.

\section{Proofs of Main Results}

\subsection{Proof of Theorem \ref{thm:weakconvergence d dim}}

\begin{lemma}\label{lem:tightness d dim}
Let Assumption \ref{ass:weak_conv} (ii) hold.
Then \((\mathbb P^n)_{n\ge1}\) is tight in $\mathcal M_1\bigl(D([0,1];\mathbb R^d)\bigr)$.
\end{lemma}

\begin{proof}
Let \(X_t(\omega)=\omega_t\) be the canonical process. We verify Aldous' criterion.
For \(r\ge1\), set
$
\tau_r:=\inf\{t\in[0,1]: |X_t|\ge r\}\wedge 1.
$
We first identify the stopped process \(X^{\tau_r}\) as a square-integrable martingale with
predictable quadratic variation controlled by \(\Sigma_1\).

Choose \(\theta\in C_b^2(\mathbb R)\) such that \(\theta(y)=y\) on \([-1,1]\), and define
\[
\theta_R(y):=R\theta(y/R), \qquad \varphi_{i,R}(x):=\theta_R(x_i).
\]
Then \(\varphi_{i,R}\in C_b^2(\mathbb R^d)\), \(\varphi_{i,R}(x)=x_i\) whenever \(|x_i|\le R\), and
\(\|\theta_R''\|_\infty\lesssim R^{-1}\). Since \(\mathbb P^n\) solves the martingale problem for
\(\mathcal L^n\), the process
\[
\varphi_{i,R}(X_t)-\int_0^t (\mathcal L_s^n\varphi_{i,R})(X_s)\,ds
\]
is a \(\mathbb P^n\)-martingale. Using the centering of the jumps and a second-order Taylor expansion,
\[
\sup_{(t,x)}|(\mathcal L_t^n\varphi_{i,R})(x)|\lesssim \frac{1}{R}\sup_{(t,x)}\Sigma_{ii}(t,x)
\le \frac{CM}{R}.
\]
Hence, for fixed \(r\) and \(R>r\), optional stopping at \(t\wedge\tau_r\) yields
$
X^i_{t\wedge\tau_r}
=
\varphi_{i,R}(X_{t\wedge\tau_r}),
$
and letting \(R\to\infty\) shows that \(X^i_{t\wedge\tau_r}\) is a martingale. Thus each coordinate
process \(X^i\) is a local martingale.

Next choose \(\kappa\in C_b^2(\mathbb R)\) such that \(\kappa(y)=y^2\) on \([-1,1]\), and define
\[
\kappa_R(y):=R^2\kappa(y/R), \qquad \psi_{i,R}(x):=\kappa_R(x_i).
\]
Then \(\psi_{i,R}(x)=x_i^2\) on \(\{|x_i|\le R\}\), while \(\sup_R\|\kappa_R''\|_\infty<\infty\) and
\(\kappa_R''(y)\to 2\) for each fixed \(y\). Again,
\[
\psi_{i,R}(X_t)-\int_0^t (\mathcal L_s^n\psi_{i,R})(X_s)\,ds
\]
is a martingale. By Taylor's formula,
$
|(\mathcal L_t^n\psi_{i,R})(x)|\lesssim \Sigma_{ii}(t,x)\le M,
$
and for fixed \((t,x)\),
\[
(\mathcal L_t^n\psi_{i,R})(x)\longrightarrow \Sigma_{ii}(t,x)
\qquad (R\to\infty).
\]
Therefore, for \(R>r\), optional stopping and dominated convergence give that
$$
(X^i_{t\wedge\tau_r})^2-\int_0^{t\wedge\tau_r}\Sigma_{ii}(s,X_s)\,ds
$$
is a martingale. Summing over \(i\) yields
\[
|X_{t\wedge\tau_r}|^2-\int_0^{t\wedge\tau_r}\tr(\Sigma(s,X_s))\,ds
\]
is a martingale.

Now let \(\tau\) be any stopping time with \(\tau\le 1-\delta\). Since \(X^{\tau_r}\) is a martingale and $\Sigma$ is bounded under Assumption \ref{ass:weak_conv} (ii),
\[
\mathbb E^{\mathbb P^n}\!\left[
|X_{(\tau+\delta)\wedge\tau_r}-X_{\tau\wedge\tau_r}|^2
\right]
=
\mathbb E^{\mathbb P^n}\!\left[
\int_{\tau\wedge\tau_r}^{(\tau+\delta)\wedge\tau_r}
\tr(\Sigma(s,X_s))\,ds
\right]
\le M\,\delta.
\]
Hence, by Markov's inequality,
$
\mathbb P^n\!\left(
|X_{(\tau+\delta)\wedge\tau_r}-X_{\tau\wedge\tau_r}|>\eta
\right)
\le \frac{M\,\delta}{\eta^2}.
$
To remove the localization, note that
\[
\{|X_{\tau+\delta}-X_\tau|>\eta\}
\subset
\{|X_{(\tau+\delta)\wedge\tau_r}-X_{\tau\wedge\tau_r}|>\eta\}
\cup
\{\tau_r\le 1\},
\]
so
$
\mathbb P^n\!\left(|X_{\tau+\delta}-X_\tau|>\eta\right)
\le \frac{M\,\delta}{\eta^2}+\mathbb P^n(\tau_r\le1).
$
It remains to control $\mathbb P^n(\tau_r\le1)$ uniformly in $n$. Since $X_0=x_0$,
\[
\mathbb E^{\mathbb P^n}\!\left[|X_{1\wedge\tau_r}|^2\right]
=
|x_0|^2+
\mathbb E^{\mathbb P^n}\!\left[\int_0^{1\wedge\tau_r}\tr(\Sigma(s,X_s))\,ds\right]
\le |x_0|^2+M.
\]
Applying Doob's inequality to the square-integrable martingale $X^{\tau_r}$, we get
\[
\mathbb P^n(\tau_r\le1)
\le
\mathbb P^n\!\left(\sup_{t\le1}|X_{t\wedge\tau_r}|\ge r\right)
\le
\frac{4}{r^2}\,
\mathbb E^{\mathbb P^n}\!\left[|X_{1\wedge\tau_r}|^2\right]
\le
\frac{4(|x_0|^2+M)}{r^2}.
\]
Therefore
\[
\sup_{n\ge1}\sup_{\tau\le1-\delta}
\mathbb P^n\!\left(|X_{\tau+\delta}-X_\tau|>\eta\right)
\le
\frac{M\,\delta}{\eta^2}+\frac{4(|x_0|^2+M)}{r^2}.
\]
Letting first \(\delta\downarrow0\), then \(r\to\infty\), proves Aldous' condition. Since \(X_0\equiv x_0\),
the initial laws are tight. Therefore \((\mathbb P^n)_{n\ge1}\) is tight in $\mathcal M_1\bigl(D([0,1];\mathbb R^d)\bigr)$.
\end{proof}

\begin{lemma}\label{lem:continuous limit}
Let Assumption \ref{ass:weak_conv} (iii) hold. For every \(K>0\), 
$
\P^n\Big(\sup_{t\le 1}|\Delta X_t|>K\Big)\xrightarrow[n\to\infty]{} 0.
$
\end{lemma}

\begin{proof}
Fix $K>0$ and let $N^{K}:=\sum_{0<t\le 1}\mathbf 1_{\{|\Delta X_t|>K\}}$ be the number of jumps of $X_t$ on $[0,1]$ whose size exceeds $K$.
Then
\[
{\mathbb P^n}\Big(\sup_{t\le 1}|\Delta X_t|>K\Big)
=
{\mathbb P^n}(N^{K}\ge 1)
\le
\mathbb E^{\mathbb P^n}[N^{K}].
\]

\noindent
By the compensator formula for marked point processes,
\[
\begin{aligned}
\mathbb{E}^{\mathbb{P}^n}[N^{K}]
&=
\mathbb{E}^{\mathbb{P}^n}\!\left[
\int_0^1
\int_{\mathbb{R}^d}
\int_0^\infty
\mathbf{1}_{\left\{
\left|\frac{1}{\sqrt{n}}\bar\Sigma(t, X_{t-})^{1/2}e\right|>K
\right\}}
\mathbf{1}_{\{0\le u\le n \lambda(t, X_{t-})\}}
\,du\,\eta(de)\,dt
\right] \\
&=
\mathbb{E}^{\mathbb{P}^n}\!\left[
\int_0^1
n \lambda(t, X_{t-})
\int_{\mathbb{R}^d}
\mathbf{1}_{\left\{
\left|\frac{1}{\sqrt{n}}\bar\Sigma(t, X_{t-})^{1/2}e\right|>K
\right\}}
\,\eta(de)\,dt
\right].
\end{aligned}
\]
\noindent
Using Markov's inequality, we obtain
\[
\int_{\mathbb{R}^d}
\mathbf 1_{\left\{
\left|\frac1{\sqrt n}\bar\Sigma(t, X_{t-})^{1/2}e\right|>K
\right\}}
\,\eta(de)
\le
\frac{1}{K^{2+\varepsilon}}
\int_{\mathbb{R}^d}
\left|
\frac1{\sqrt n}\bar\Sigma(t, X_{t-})^{1/2}e
\right|^{2+\varepsilon}
\eta(de).
\]
Hence
\[
\mathbb E^{\mathbb P^n}[N^{K}]
\le
\frac{1}{K^{2+\varepsilon}n^{\varepsilon/2}}
\mathbb E^{\mathbb P^n}\!\left[
\int_0^1
\lambda(t, X_{t-})
\int_{\mathbb{R}^d}
|\bar\Sigma(t, X_{t-})^{1/2}e|^{2+\varepsilon}\,\eta(de)\,dt
\right].
\]
\noindent
Since $\|\bar\Sigma(t, X_{t-})\|_{\mathrm{op}}\le \tr(\bar\Sigma(t, X_{t-}))$, it follows that
\[
\int_{\mathbb{R}^d}
|\bar\Sigma(t, X_{t-})^{1/2}e|^{2+\varepsilon}\,\eta(de)
\le
\tr(\bar\Sigma(t, X_{t-}))^{1+\varepsilon/2}\int_{\mathbb{R}^d}|e|^{2+\varepsilon}\,\eta(de).
\]
Therefore
\[
\mathbb E^{\mathbb P^n}[N^{K}]
\le
\frac{1}{K^{2+\varepsilon}n^{\varepsilon/2}}
\left(\int_{\mathbb{R}^d}|e|^{2+\varepsilon}\,\eta(de)\right)
\sup_{\mathbb P^n \in \mathcal M}\mathbb E^{\mathbb P^n}\!\left[
\int_0^1
\tr(\bar\Sigma(t, X_{t-}))^{1+\varepsilon/2}\,\lambda(t, X_{t-})\,dt\right].
\]
Under \ref{ass:weak_conv} (i) and \ref{ass:weak_conv} (iii), this tends to $0$ as $n\to\infty$, which yields the result.
\end{proof}

\begin{lemma}[Uniform bound on generators]\label{lem:Lu_rate_non_gaussian}
Let \(\varepsilon\in(0,1)\) and \(f\in C_b^{2,\varepsilon}(\mathbb R^d)\).
Then for all \((u,x)\in[0,1]\times\mathbb R^d\),
\[
\left|
(\mathcal L_u^n f)(x)-(\mathcal L_u f)(x)
\right|
\le
C\, \lambda(u, x)\,\tr(\bar\Sigma(u, x))^{1+\varepsilon/2}\,n^{-\varepsilon/2},
\qquad
C=C(\varepsilon,f,\eta).
\]
where $\mathcal L_u$ is defined in \eqref{eq:mtgp_d} and $\mathcal L_u^n$ in \eqref{eq:L_n}.
\end{lemma}

\begin{proof}
Let $f\in C_b^{2,\varepsilon}(\mathbb R^d)$ and \(\varepsilon\in(0,1)\). Using a the second-order Taylor expansion of $f$,
\[
f(x+z)
=
f(x)+\nabla f(x)\cdot z+\frac12 z^\top D^2 f(x)z+R(x,z),
\]
with
$
|R(x,z)|\le C_\varepsilon [D^2 f]_{C^\varepsilon}|z|^{2+\varepsilon},
$
and since $\eta$
has zero mean and unit covariance matrix by Assumption \ref{ass:weak_conv} (i),
\[
(\mathcal L_u^n f)(x)
=
\frac12\mathrm{\tr}\!\big(D^2 f(x)\Sigma_u(x)\big)
+
n \lambda(u,  x)
\int_{\mathbb{R}^d}
R\!\left(x,\frac1{\sqrt n}\bar\Sigma(u, x)^{1/2}e\right)
\eta(de).
\]
\noindent
Hence
\[
\left|
(\mathcal L_u^n f)(x)-(\mathcal L_u f)(x)
\right|
\le
C_\varepsilon [D^2 f]_{C^\varepsilon}\,
n \lambda(u,  x)
\int_{\mathbb{R}^d}
\left|
\frac1{\sqrt n}\bar\Sigma(u, x)^{1/2}e
\right|^{2+\varepsilon}
\eta(de).
\]
Since
$
|\bar\Sigma(u, x)^{1/2}e|^{2+\varepsilon}
\le
\tr(\bar\Sigma(u, x))^{1+\varepsilon/2}|e|^{2+\varepsilon},
$
it follows that
\[
\left|
(\mathcal L_u^n f)(x)-(\mathcal L_u f)(x)
\right|
\le
C_\varepsilon [D^2 f]_{C^\varepsilon}\,
n^{-\varepsilon/2}
\,\lambda(u,  x)\,\tr(\bar\Sigma(u, x))^{1+\varepsilon/2}
\int_{\mathbb{R}^d} |e|^{2+\varepsilon}\,\eta(de).
\]
\noindent
By Assumption \ref{ass:weak_conv} (i), there exists a constant $C^{\prime}$ such that
$
\int_{\mathbb{R}^d} |e|^{2+\varepsilon}\,\eta(de) \le C^{\prime},
$
which yields the claim.
\end{proof}

\begin{lemma}
\label{lem:mtgproperty d dim}
Assume that \ref{ass:weak_conv} (ii) and \ref{ass:weak_conv} (iii) hold.
Suppose \(\mathbb P\in\mathcal M_1(C([0,1],\mathbb R^d))\) is a weak limit of
\(\Law(X)\). Then \(\mathbb P\) solves the martingale problem for
\eqref{eq:mtgp_d}.
\end{lemma}

\begin{proof}
Let $f\in C_b^{2,\varepsilon}(\mathbb R^d)$. 
\medskip
\noindent
Fix $0\le s\le t\le1$. 
\noindent
Define
\[
M_t^{n,f}
:=
f(X_t)-f(X_0)-\int_0^t (\mathcal L_u^n f)(X_u)\,du .
\]
By It\^o's formula for jump processes, $M_t^{n,f}$ is a martingale under $\mathbb P^n$.
Let $0\le t_1<\dots<t_k\le s$ and let
$h\in C_b((\mathbb R^d)^k)$,
$H:=h(X_{t_1},\dots,X_{t_k}).$
Then
\begin{equation}\label{eq:quenched-prelimit}
\mathbb E_{\mathbb P^n}
\!\left[
\Big(
f(X_t)-f(X_s)
-\int_s^t (\mathcal L_u^n f)(X_u)\,du
\Big)
H
\right]
=0 .
\end{equation}
\noindent
Subtracting the limit generator yields
\begin{align}\label{eq:defect-random}
&\mathbb E_{\mathbb P^n}
\!\left[
\Big(
f(X_t)-f(X_s)
-\int_s^t (\mathcal L_u f)(X_u)\,du
\Big)
H
\right]
\\
&\qquad=
\mathbb E_{\mathbb P^n}
\!\left[
\left(
\int_s^t
\big[(\mathcal L_u^n-\mathcal L_u)f\big](X_u)\,du
\right)
H
\right].
\nonumber
\end{align}

\noindent
By Lemma~\ref{lem:Lu_rate_non_gaussian}, there exists
$C=C(\varepsilon,f,\eta)$ such that
\[
\left|
(\mathcal L_u^n f-\mathcal L_u f)(X_u)
\right|
\le
C\,\lambda(u,  X_u)\,\tr(\bar\Sigma(u, X_u))^{1+\varepsilon/2}\,n^{-\varepsilon/2}.
\]
We obtain
\[
\left|
\mathbb E_{\mathbb P^n}
\!\left[
\left(
\int_s^t
[(\mathcal L_u^n-\mathcal L_u)f](X_u)\,du
\right)
H
\right]
\right|
\le
C\,\|H\|_\infty\,n^{-\varepsilon/2}
\,
\sup_{\mathbb P^n \in \mathcal M}\mathbb E_{\mathbb P^n}
\!\left[
\int_s^t \tr(\bar\Sigma(u, X_u))^{1+\varepsilon/2}\,\lambda(u,  X_u)\,du
\right].
\]
\noindent
By
\ref{ass:weak_conv} (iii), the right-hand side
converges to $0$ as $n\to\infty$. Hence
\begin{equation}\label{eq:quenched-limit}
\mathbb E_{\mathbb P^n}
\!\left[
\Big(
f(X_t)-f(X_s)
-\int_s^t (\mathcal L_u f)(X_u)\,du
\Big)
H
\right]
\longrightarrow 0 .
\end{equation}
\noindent
Define a bounded functional $F:D([0,1],\mathbb R^d)\to\mathbb R$ by
\[
F(\omega)
=
\Big(
f(\omega_t)-f(\omega_s)
-\int_s^t (\mathcal L_u f)(\omega_u)\,du
\Big)
h(\omega_{t_1},\dots, \omega_{t_k}).
\]
\noindent
Since $\mathcal L_u f$ is bounded and continuous, $F$ is continuous at every continuous path $\omega$. Indeed, if $\omega_n\to\omega$ in the Skorokhod topology and $\omega$ is continuous, then $\omega_n\to\omega$ uniformly on $[0,1]$, so all point evaluations and the time integral converge.
\noindent
Now $\mathbb P^n\Rightarrow \mathbb P$ and $\mathbb P(C([0,1],\mathbb R^d))=1$ yields
$
\mathbb E^{\mathbb P}[F]=\lim\limits_{n \to \infty}\mathbb E^{\mathbb P^n}[F]=0.
$
Therefore
$
M_t^f
:=
f(X_t)
- f(X_0) -
\int_0^t
\frac12
\mathrm{\tr}\big(\Sigma(u, X_u) D^2 f(X_u)\big)
\,du
$
is a martingale under $\mathbb P$, so that $\mathbb P$ solves the
martingale problem for  $\mathcal L\varphi$.
\end{proof}

\begin{proof}[Proof of Theorem \ref{thm:weakconvergence d dim}]
Since the martingale problem for \(\mathcal L\) is well posed under Assumption \ref{ass:weak_conv} (ii) (for instance by Stroock--Varadhan theory for bounded, continuous, uniformly elliptic coefficients; see \cite{stroock-varadhan-1979}), every limit point must coincide with $\mathbb P$. As tightness has already been established, this identifies the unique weak limit of the full sequence and completes the proof of Theorem \ref{thm:weakconvergence d dim}. 
\end{proof}

\subsection{Proof of Theorem \ref{thm:limit of chaos}}

\begin{proposition}
\label{prop:entropy_formula}
Under Assumption \ref{ass:chaos_limit} (i), the relative entropy satisfies
\begin{align}
H(\mathbb P_1^n\|\mathbb P_2^n)
&=
\mathbb E^{\mathbb P_1^n}\!\left[
\int_0^1
\left(
\lambda_1^n(t,X_{t-})
\log\frac{\lambda_1^n(t,X_{t-})}{\lambda_2^n(t,X_{t-})}
-\lambda_1^n(t,X_{t-})
+\lambda_2^n(t,X_{t-})
\right)\,dt
\right]
\nonumber\\
&\quad
+
\mathbb E^{\mathbb P_1^n}\!\left[
\int_0^1
\lambda_1^n(t,X_{t-})\,
H\!\left(
\eta_1^n(t,X_{t-},\cdot)\,\middle\|\,\eta_2^n(t,X_{t-},\cdot)
\right)\,dt
\right].
\label{eq:general_entropy_marked}
\end{align}
\end{proposition}

\begin{proof}
By the Girsanov formula for marked point processes; see Jacod \cite[Th.~5.1]{Jacod1975}, the log-likelihood ratio is
\begin{align*}
\log\frac{d \mathbb P_1^n}{d\mathbb P_2^n}
&=
\int_0^1\int_{\mathbb R^d}
\log\!\left(
\frac{\lambda_1^n(t,X_{t-})\,\eta_1^n(t,X_{t-},dz)}
{\lambda_2^n(t,X_{t-})\,\eta_2^n(t,X_{t-},dz)}
\right)\mu(dt,dz)
\\
&\quad
-
\int_0^1\int_{\mathbb R^d}
\Big(
\lambda_1^n(t,X_{t-})\,\eta_1^n(t,X_{t-},dz)
-
\lambda_2^n(t,X_{t-})\,\eta_2^n(t,X_{t-},dz)
\Big)\,dt.
\end{align*}
Since \(\eta_1^n(t,x,\cdot)\) and \(\eta_2^n(t,x,\cdot)\) are probability measures, the compensator term reduces to
\[
\int_0^1
\big(
\lambda_1^n(t,X_{t-})-\lambda_2^n(t,X_{t-})
\big)\,dt.
\]
Moreover,
$
\log\!\left(
\frac{\lambda_1^n\,\eta_1^n}{\lambda_2^n\,\eta_2^n}
\right)
=
\log\frac{\lambda_1^n}{\lambda_2^n}
+
\log\frac{d\eta_1^n}{d\eta_2^n},
$
so that
\begin{align*}
\log\frac{d \mathbb P_1^n}{d\mathbb P_2^n}
&=
\int_0^1
\log\frac{\lambda_1^n(t,X_{t-})}{\lambda_2^n(t,X_{t-})}\,dN_t
-
\int_0^1
\big(
\lambda_1^n(t,X_{t-})-\lambda_2^n(t,X_{t-})
\big)\,dt
\\
&\quad
+
\int_0^1\int_{\mathbb R^d}
\log\frac{d\eta_1^n(t,X_{t-},\cdot)}{d\eta_2^n(t,X_{t-},\cdot)}(z)\,
\mu(dt,dz),
\end{align*}
where \(N_t=\mu([0,t]\times\mathbb R^d)\) is the jump counting process.
Taking expectation under \(\mathbb \mathbb P_1^n\) and using the compensation formula gives
\begin{align*}
H( \mathbb P_1^n\| \mathbb P_2^n)
&=
\mathbb E^{ \mathbb P_1^n}\!\left[
\int_0^1
\lambda_1^n(t,X_{t-})
\log\frac{\lambda_1^n(t,X_{t-})}{\lambda_2^n(t,X_{t-})}\,dt
\right]
\\
&\quad
-
\mathbb E^{\mathbb P_1^n}\!\left[
\int_0^1
\big(
\lambda_1^n(t,X_{t-})-\lambda_2^n(t,X_{t-})
\big)\,dt
\right]
\\
&\quad
+
\mathbb E^{\mathbb P_1^n}\!\left[
\int_0^1
\lambda_1^n(t,X_{t-})
\int_{\mathbb R^d}
\log\frac{d\eta_1^n(t,X_{t-},\cdot)}{d\eta_2^n(t,X_{t-},\cdot)}(z)\,
\eta_1^n(t,X_{t-},dz)\,dt
\right].
\end{align*}
The inner integral is exactly
$
H\!\left(
\eta_1^n(t,X_{t-},\cdot)\,\middle\|\,\eta_2^n(t,X_{t-},\cdot)
\right),
$
which yields \eqref{eq:general_entropy_marked}.
\end{proof}

\begin{proof}[Proof of Theorem \ref{thm:limit of chaos}]
We specialize the entropy identity \eqref{eq:general_entropy_marked} to the present Gaussian framework. Using the explicit formula for the relative entropy between centered Gaussian laws, we obtain
\[ \begin{aligned} H(\mathbb P_1^n \| \mathbb P_2^n) &= n\,\mathbb{E}^{\mathbb P_1^n}\int_0^1 \Biggl[ \lambda_1(t,X_{t-}) \log \frac{\lambda_1(t,X_{t-})}{\lambda_2(t,X_{t-})} -\lambda_1(t,X_{t-}) +\lambda_2(t,X_{t-}) \\ &\qquad\qquad +\frac{\lambda_1(t,X_{t-})}{2}\, \Phi\!\left( \bar\Sigma_2^{-1}(t,X_{t-})\,\bar\Sigma_1(t,X_{t-}) \right) \Biggr]\,dt, \end{aligned} \]
where
$
\Phi(A):=\tr(A)-d-\log\det(A).
$

Dividing by \(n\), and using that \(X_t=X_{t-}\) for Lebesgue-a.e.\ \(t\), we can rewrite the normalized entropy as
\[
\frac1n H(\mathbb P_1^n\|\mathbb P_2^n)
=
\E^{\mathbb P_1^n}\!\left[\int_0^1 \ell(t,X_t)\,dt\right].
\]
Recall that
$
\Psi(\omega):=\int_0^1 \ell(t,\omega_t)\,dt,
\ \omega\in D([0,1],\R^d).
$
Then
$
\frac1n H(\mathbb P_1^n\|\mathbb P_2^n)=\E^{\mathbb P_1^n}[\Psi(X)].
$
By Theorem~\ref{thm:weakconvergence d dim}, we have
\[
\mathbb P_1^n \Rightarrow \P
\qquad\text{on } D([0,1],\R^d),
\]
equipped with the Skorokhod \(J_1\)-topology. Moreover, by Lemma~\ref{lem:continuous limit}, the limiting law \(\P\) is supported on \(C([0,1],\R^d)\).
Since \(\ell\) is continuous by Assumption~\ref{ass:chaos_limit} (ii), the mapping
$
\omega \longmapsto \Psi(\omega)=\int_0^1 \ell(t,\omega_t)\,dt
$
is continuous at every \(\omega\in C([0,1],\R^d)\). Therefore, \(\Psi\) is \(\P\)-a.s.\ continuous. Combined with the weak convergence \(\mathbb P_1^n\Rightarrow \P\), this yields convergence in distribution of \(\Psi(X)\) under \(\mathbb P_1^n\) toward \(\Psi(X)\) under \(\P\). Finally, Assumption~\ref{ass:chaos_limit} (iii) gives the uniform integrability needed to pass to expectations. We conclude that
\[
\E^{\mathbb P_1^n}\!\left[\int_0^1 \ell(t,X_t)\,dt\right]
\longrightarrow
\E^{\P}\!\left[\int_0^1 \ell(t,X_t)\,dt\right].
\]
Therefore,
$
\lim_{n\to\infty}\frac1n H(\mathbb P_1^n\|\mathbb P_2^n)
=
\E^{\P}\!\left[\int_0^1 \ell(t,X_t)\,dt\right].
$
This is exactly \eqref{eq:se}.
\end{proof}

\subsection{Proof of Proposition~\ref{corr:LDP-Rn}}

\begin{proof}
\noindent\emph{Step 1.}
Under the $1$-dim Wiener measure $\mathbb P$, the two increments
\[
\omega(i/n)-\omega((2i-1)/(2n)),
\ \
\omega((2i-1)/(2n))-\omega((i-1)/n)
\]
are independent centered Gaussian random variables with variance $1/(2n)$. Therefore the sum of their square is exponential distributed with rate $n$. Hence, $S^n_k$ are arrival times of a renewal process with i.i.d. exponential interarrival times of rate $n$, and 
 $(N^n_t)$ is a Poisson process with rate $n$ and jump size $1/n$.

According to \cite[Exercise 5.2.12]{DeZe10}, the scaled Poisson process $(N^n)_{n\ge1}$ satisfies a large
deviation principle on $D([0,1],\mathbb R_+)$, endowed with the $J_1$-topology,
with speed $n$ and good rate function
\[
J(a)
=
\begin{cases}
\displaystyle
\int_0^1 \ell(\dot a(t))\,dt,
&
\text{if } a\in A,
\\[1.2ex]
+\infty,
&
\text{otherwise},
\end{cases}
\]
where $\ell(q)=q\log q-q+1$ for $q\ge0$. 

\medskip

\noindent
\emph{Step 2.}  For $a\in D([0,1],\mathbb R_+)$, denote by $X_{\cdot}$ the canonical process on $D:=D([0,\infty); \mathbb R^d)$,
\[
\Phi_n(a)
:=
(X_{\cdot} \mapsto (X_{a(t)})_{t \in [0,1]})_{\#}\mathbb P^n \in\mathcal M_1(D),
\qquad
\Phi(a)
:=(X_{\cdot} \mapsto (X_{a(t)})_{t \in [0,1]} )_{\#} \mathbb W\in \mathcal{P}(D).
\]
Then by definition,
\( 
R_n=\Phi_n(N^n).
\) We claim that if $a_n\to a$ in $D([0,1],\mathbb R_+)$ and $J(a)<\infty$, then
$
\Phi_n(a_n) \to \Phi(a)$ in weak topology.
Indeed, if $J(a)<\infty$, then $a$ is absolutely continuous.
Since convergence in $J_1$ to a continuous limit is equivalent to uniform
convergence, we have
$
\sup_{t\in[0,1]}|a_n(t)-a(t)|\to0.
$
Fix $T>\sup_n a_n(1)\vee a(1)$. According to Theorem~\ref{thm:weakconvergence d dim}, 
$\mathbb P^n$ converges to $\mathbb W$ in weak topology. By the Skorokhod representation theorem, after
possibly enlarging the probability space, we may assume that there exist processes $Z^n_{\cdot} \sim \mathbb P^n$, and
$
Z^n\to B$ almost surely in $D([0,T],\mathbb R^d)$.
As the limit $B$ is continuous, this convergence is actually uniform on
$[0,T]$, i.e.
$
\sup_{s\in[0,T]}|Z_s^n-B_s|\to0$ a.s.
Therefore, we have the estimate
\begin{align*}
\sup_{t\in[0,1]}|Z^n_{a_n(t)}-B_{a(t)}|
&\le
\sup_{s\in[0,T]}|Z_s^n-B_s|
+
\sup_{t\in[0,1]}|B_{a_n(t)}-B_{a(t)}|.
\end{align*}
The first term tends to $0$ almost surely, and the second also tends to $0$
almost surely because Brownian paths are uniformly continuous on $[0,T]$ and
$a_n\to a$ uniformly. Hence
\[
Z^n\circ a_n \to B\circ a
\qquad\text{almost surely in }D([0,1],\mathbb R),
\]
which implies
\(
\Phi_n(a_n)\to\Phi(a).
\) in weak topology. 

\medskip

\noindent
\emph{Step 3.}
Since $(N^n)$ satisfies an LDP on $D([0,1],\mathbb R_+)$ with speed $n$ and rate 
$J$, and since whenever $a_n\to a$ with $J(a)<\infty$ we have
$
\Phi_n(a_n)\to \Phi(a),
$
the extended contraction principle, \cite[Theorem 4.2.23]{DeZe10}, yields that $(R_n)$ satisfies an LDP on
$\mathcal M_1(D)$ with speed $n$ and rate function
$
I(P)
=
\inf\{J(a):\ \Phi(a)=P\},
$
which is exactly \eqref{eq:ratefunction}.

\end{proof}

\subsection{Proof of Theorem~\ref{theorem:strong duality}}

To establish the duality result, we follow the approach of \cite{TanTouzi2013} and reuse their main results.
\noindent
We begin with some preliminary results. Let $(\mu_0^n)$ and $(\mu_1^n)$ be two sequences in $\mathcal M(\mathbb R^d)$
converging weakly to $\mu_0, \mu_1 \in \mathcal M(\mathbb R^d)$, respectively.
We restrict to the case $\liminf_{n\to\infty} V(\mu_0^n,\mu_1^n)<\infty$, since the alternative case is trivial. Up to extracting a subsequence, we may assume that $\sup_{n\ge1} V(\mu_0^n,\mu_1^n)<\infty$. We can then select a sequence $(\mathbb P^n)_{n\ge1}$ with $\mathbb P^n\in\mathcal P(\mu_0^n,\mu_1^n)$ such that: 
\begin{equation}
\sup_{n \ge 1} J(\mathbb{P}^n) < \infty,
\qquad
0 \le J(\mathbb{P}^n) - V(\mu_0^n,\mu_1^n) \longrightarrow 0
\quad \text{as } n \to \infty.
\end{equation}

\begin{lemma}\label{lem:33}
For each $t$ in [0,1], define
$
\Sigma_1^n(t,x):=\Sigma_1^{\mathbb P^n}(t,x)
$
and 
$
\lambda_1^n(t,x):=\frac1d\tr\!\big(\bar\Sigma_2^{-1}(t,x)\Sigma_1^n\big).
$
Then
\begin{equation}\label{eq:entropy_integrability}
\sup_{n\ge1}\,
\mathbb E^{\mathbb P^n}\!\left[\int_0^1 \lambda_1^n(t, X_t) \log \lambda_1^n(t, X_t)\,dt\right]
<\infty.
\end{equation}
\end{lemma}
\begin{proof}
Since \(\xi\) is bounded, with \(C_\xi:=\|\xi\|_\infty\), we have
\[
\mathbb E^{\mathbb P^n}\!\left[\int_0^1 \ell^{\mathrm{tr}}(t, X_t, \Sigma_{1,t}^n)\,dt\right]
=
J(\mathbb P^n)-\mathbb E^{\mathbb P^n}\!\left[\xi(X_{1\wedge\cdot})\right]
\le
J(\mathbb P^n)+C_\xi.
\]
Hence
$
\sup_{n\ge1}
\mathbb E^{\mathbb P^n}\!\left[\int_0^1 \ell^{\mathrm{tr}}(t, X_t, \Sigma_{1,t}^n)\,dt\right]
\le
\sup_{n\ge1} J(\mathbb P^n)+C_\xi
<\infty.
$
It is therefore enough to derive a lower bound on \(\ell^{\mathrm{tr}}(t,x, \Sigma_1)\) in terms of
\(\lambda_1\log \lambda_1\), uniformly in \(t\in[0,1]\).
Fix \(t\in[0,1]\), and let \(\Sigma_1\in\mathbb S^d_+\). 
We first consider the case \(\Sigma_1\in\mathbb S_{++}^d\). Let's define
\[
B:=\bar\Sigma_2^{-1/2}(t,x)\,\bar\Sigma_1\,\bar\Sigma_2^{-1/2}(t,x)\in\mathbb S_{++}^d.
\]
Then
$
\tr(B)
=
\tr\!\big(\bar\Sigma_2^{-1}(t,x)\bar\Sigma_1\big)
=
\frac{1}{\lambda_1}\tr\!\big(\bar\Sigma_2^{-1}(t,x),\Sigma_1\big)
=
\frac{d\lambda_1}{\lambda_1}
=
d.
$
Applying the arithmetic-geometric mean inequality to the eigenvalues of \(B\), we obtain
$
\det(B)\le \Big(\frac{\tr(B)}{d}\Big)^d=1.
$
Hence
\[
\log\det\!\big(\bar\Sigma_2^{-1}(t)\bar\Sigma_1\big)
=
\log\det(B)
\le 0,
\]
and therefore
$
-\frac {\lambda_1}{2}\log\det\!\big(\bar\Sigma_2^{-1}(t,x)\bar\Sigma_1\big)\ge0.
$
It follows from the definition of \(\ell^{\mathrm{tr}}\) that
\[
\ell^{\mathrm{tr}}(t, x, \Sigma_1)
\ge
\lambda_1\log\frac{\lambda_1}{\lambda_2(t,x)}-\lambda_1+\lambda_2(t,x)
\ge
\lambda_1\log \lambda_1 - \lambda_1\log \lambda_2(t,x) - \lambda_1.
\]
Next, since \(0<\underline b\le \lambda_2(t,x)\le \overline b\) for all \(t\in[0,1]\), there exists a constant
$
K:=1+\max\{|\log \underline b|,\,|\log \overline b|\}
$
such that
$
|\log \lambda_2(t,x)|\le K-1,
\ t\in[0,1].
$
Thus, for every \(\Sigma_1\in\mathbb S_{++}^d\),
$
\ell^{\mathrm{tr}}(t, x, \Sigma_1)\ge \lambda_1\log \lambda_1-K\lambda_1.
$
The same inequality also holds for \(\Sigma_1=0\), because then \(\lambda_1=0\) and
$
\ell^{\mathrm{tr}}(t, x, 0)=\lambda_2(t,x)\ge0.
$
For \(\Sigma_1\in\partial\mathbb S^d_+\setminus\{0\}\), the inequality is trivial since
$
\ell^{\mathrm{tr}}(t,x,\Sigma_1)=+\infty.
$
Therefore, for all \(\Sigma_1\in\mathbb S^d_+\),
$
\ell^{\mathrm{tr}}(t,x, \Sigma_1)\ge \lambda_1\log \lambda_1-K\lambda_1.
$
Now consider the function
\[
f_K(x):=\frac12 x\log x-Kx,
\qquad x\ge0,
\]
with the convention \(0\log0:=0\). Since \(f_K\) is continuous on \([0,\infty)\) and
\(f_K(x)\to+\infty\) as \(x\to\infty\), there exists a constant \(C_K<\infty\) such that
$
\frac12 x\log x-Kx\ge -C_K,
\qquad x\ge0.
$
Equivalently,
$
x\log x-Kx\ge \frac12 x\log x-C_K$, for $x\ge0$.
Combining this with the previous estimate yields
\[
\ell^{\mathrm{tr}}(t, x, \Sigma_1)\ge \frac12\, \lambda_1\log \lambda_1 - C_K,
\qquad \Sigma_1\in\mathbb S^d_+,\ t\in[0,1].
\]
Applying this bound to \(\Sigma_1=\Sigma_1^n(t,x)\), integrating over \([0,1]\), and taking expectations, we get
\[
\frac12\,
\mathbb E^{\mathbb P^n}\!\left[\int_0^1 \lambda_1^n(t, X_t)\log \lambda_1^n(t, X_t)\,dt\right]
\le
\mathbb E^{\mathbb P^n}\!\left[\int_0^1 \ell^{\mathrm{tr}}(t, X_t, \Sigma_{1,t}^n)\,dt\right]+C_K.
\]
Taking the supremum over \(n\) and using the uniform bound above on
\(\mathbb E^{\mathbb P^n}[\int_0^1 \ell^{\mathrm{tr}}(t, X_t, \Sigma_{1,t}^n)\,dt]\) yields the result.
\end{proof}

% \begin{proposition}\label{prop:L properties}
% The function \(\ell^{\mathrm{tr}}\) is
% nonnegative, convex, and lower semicontinuous on \(\mathbb S_+^d\).
% \end{proposition}

\begin{proposition}\label{prop:L properties}
Under Assumption \ref{ass:ref_process_duality}, the function $\ell^{\mathrm{tr}}$ is nonnegative, convex on
$\mathbb S_+^d$ and 
is lower semicontinuous with respect to the joint variable $(t,x,\Sigma_1)$.
\end{proposition}

\begin{proof}

\noindent\emph{(i) Joint lower semicontinuity.} 
Let
$
(t_n,x_n,\Sigma_{1,n})\to(t,x,\Sigma_1)
$
in $[0,1]\times\Omega\times\mathbb S_+^d$.

Set
$
\lambda_{1,n}
:=\frac1d\tr\!\bigl(\bar\Sigma_2(t_n,x_n)^{-1}\Sigma_{1,n}\bigr),
\
\lambda_1:=\lambda_1(t,x),
$
and
$
\lambda_{2,n}:=\lambda_2(t_n,x_n),
\
\lambda_2:=\lambda_2(t,x).
$
Since $(t,x)\mapsto \bar\Sigma_2(t,x)^{-1}$ is continuous and $\Sigma_{1,n}\to\Sigma_1$, we have
$
\lambda_{1,n}\to\lambda_1.
$
Also,
$
\lambda_{2,n}\to\lambda_2.
$

\smallskip
\noindent\emph{Case 1: $\Sigma_1\in\mathbb S_{++}^d$.}
Then $\Sigma_{1,n}\in\mathbb S_{++}^d$ for all $n$ large enough, and $\lambda_1>0$. Hence
\[
\bar\Sigma_{1,n}:=\frac{\Sigma_{1,n}}{\lambda_{1,n}}
\longrightarrow
\frac{\Sigma_1}{\lambda_1}
=:\bar\Sigma_1
\in\mathbb S_{++}^d.
\]
By continuity of $(t,x)\mapsto \bar\Sigma_2(t,x)^{-1}$ and of $\log\det$ on $\mathbb S_{++}^d$, we have the following convergence
\[
\log\det\!\bigl(\bar\Sigma_2(t_n,x_n)^{-1}\bar\Sigma_{1,n}\bigr)
\to
\log\det\!\bigl(\bar\Sigma_2(t,x)^{-1}\bar\Sigma_1\bigr),
\]
and therefore
$
\ell^{\mathrm{tr}}(t_n,x_n,\Sigma_{1,n})\to \ell^{\mathrm{tr}}(t,x,\Sigma_1).
$

\smallskip
\noindent\emph{Case 2: $\Sigma_1=0$.}
Then $\lambda_1=0$, hence $\lambda_{1,n}\to 0$. For every $n$,
$
\ell^{\mathrm{tr}}(t_n,x_n,\Sigma_{1,n})
\ge
\lambda_{1,n}\log\frac{\lambda_{1,n}}{\lambda_{2,n}}-\lambda_{1,n}+\lambda_{2,n},
$
because if $\Sigma_{1,n}\in\partial\mathbb S_+^d\setminus\{0\}$ the left-hand side is $+\infty$, while if
$\Sigma_{1,n}\in\mathbb S_{++}^d$ the determinant term is nonnegative by Step 1. Since $\lambda_{2,n}\to\lambda_2>0$, the sequence $(\log\lambda_{2,n})_n$ is bounded, so
$
\lambda_{1,n}\log\lambda_{2,n}\to 0.
$
Together with $\lambda_{1,n}\log\lambda_{1,n}\to 0$, this yields
$
\lambda_{1,n}\log\frac{\lambda_{1,n}}{\lambda_{2,n}}-\lambda_{1,n}+\lambda_{2,n}\to \lambda_2.
$
Thus
$
\liminf_{n\to\infty}\ell^{\mathrm{tr}}(t_n,x_n,\Sigma_n)\ge \lambda_2=\ell^{\mathrm{tr}}(t,x,0).
$

\smallskip
\noindent\emph{Case 3: $\Sigma_1\in\partial\mathbb S_+^d\setminus\{0\}$.}
Then
$
\ell^{\mathrm{tr}}(t,x,\Sigma_1)=+\infty.
$
Since $\bar\Sigma_2(t,x)^{-1}\in\mathbb S_{++}^d$ and $\Sigma_1\neq 0$,
$
\lambda_1=\frac1d\tr\!\bigl(\bar\Sigma_2(t,x)^{-1}\Sigma_1\bigr)>0,
$
hence $\lambda_{1,n}\to\lambda_1>0$. If infinitely many $\Sigma_{1,n}$ are singular and nonzero, then
$
\ell^{\mathrm{tr}}(t_n,x_n,\Sigma_{1,n})=+\infty
$
along a subsequence, and we are done. Otherwise, for $n$ large enough, $\Sigma_{1,n}\in\mathbb S_{++}^d$, and
\[
\bar\Sigma_{1,n}:=\frac{\Sigma_{1,n}}{\lambda_{1,n}}
\longrightarrow
\frac{\Sigma_1}{\lambda_1}
=:\bar\Sigma_1
\in\partial\mathbb S_+^d\setminus\{0\}.
\]
Hence
$
\det(\bar\Sigma_{1,n})\to 0,
$
so, since
$
\det\!\bigl(\bar\Sigma_2(t_n,x_n)^{-1}\bigr)\to
\det\!\bigl(\bar\Sigma_2(t,x)^{-1}\bigr)>0,
$
we get
$$
\log\det\!\bigl(\bar\Sigma_2(t_n,x_n)^{-1}\bar\Sigma_{1,n}\bigr)\to -\infty.
$$
As $\lambda_{1,n}\to\lambda_1>0$, it follows that
$$
-\frac{\lambda_{1,n}}2
\log\det\!\bigl(\bar\Sigma_2(t_n,x_n)^{-1}\bar\Sigma_{1,n}\bigr)\to +\infty,
$$
and therefore
$
\ell^{\mathrm{tr}}(t_n,x_n,\Sigma_{1,n})\to +\infty
=\ell^{\mathrm{tr}}(t,x,\Sigma_1).
$
Thus $\ell^{\mathrm{tr}}$ is jointly lower semicontinuous on
$[0,1]\times\Omega\times\mathbb S_+^d$.

\medskip
For the remaining steps, fix $(t,x)\in[0,1]\times\Omega$. For simplicity, write
$
\bar\Sigma_2=\bar\Sigma_2(t,x),
$
$
\lambda_2=\lambda_2(t,x)
$ and
$
\Sigma_1=\Sigma_1(t,x),
$
$
\lambda_1=\lambda_1(t,x).
$ 

\medskip
\emph{(ii) Nonnegativity.}
\noindent
If $\Sigma_1=0$ or $\Sigma_1\in \partial\mathbb S_+^d\setminus\{0\}$, it is trivial.
If $\Sigma_1\in\mathbb S_{++}^d$
define
$A:=\bar\Sigma_2^{-1/2}\bar\Sigma_1\,\bar\Sigma_2^{-1/2}.
$
Then \(A\in\mathbb S_{++}^d\), and
\[
\frac1d\tr(A)
=
\frac1d\tr(\bar\Sigma_2^{-1}\bar\Sigma_1)
=
\frac1{db}\tr(\bar\Sigma_2^{-1}\Sigma_1)
=1.
\]
By the arithmetic--geometric mean inequality applied to the eigenvalues of \(A\),
$
\det(A)^{1/d}\le \frac1d\tr(A)=1,
$
hence
$
\log\det(\bar\Sigma_2^{-1}\bar\Sigma_1)
=
\log\det(A)\le 0.
$
Therefore,
$
-\frac {\lambda_1}{2}\log\det(\bar\Sigma_2^{-1}\bar\Sigma_1)\ge 0.
$
On the other hand,
\[
\lambda_1\log\frac{\lambda_1}{\lambda_2}-\lambda_1+\lambda_2
=
\lambda_2\left(\frac{\lambda_1}{\lambda_2}\log\frac{\lambda_1}{\lambda_2}-\frac{\lambda_1}{\lambda_2}+1\right)\ge 0,
\]
since \(r\log r-r+1\ge 0\) for all \(r>0\). 
Altogether, \(\ell^{\mathrm{tr}}\ge 0\) on \(\mathbb S_+^d\).

\medskip
\noindent\emph{(iii) Convexity on \(\mathbb S_+^d\).}
Define
\[
\psi:[0,\infty)\to\mathbb R,
\qquad
\psi(\lambda_1):=
\begin{cases}
\lambda_1\log\frac{\lambda_1}{\lambda_2}-\lambda_1+\lambda_2, & \lambda_1>0,\\[4pt]
\lambda_2, & \lambda_1=0.
\end{cases}
\]
Since \(\lambda_1\log \lambda_1\to 0\) as \(\lambda_1\downarrow 0\), \(\psi\) is the continuous extension of
\(\lambda_1\mapsto \lambda_1\log(\lambda_1/\lambda_2)-\lambda_1+\lambda_2\), hence \(\psi\) is convex on \([0,\infty)\). Next define the extended-value function
\[
f:\mathbb S_+^d\to(-\infty,+\infty],
\qquad
f(X):=
\begin{cases}
-\dfrac12\log\det(\bar\Sigma_2^{-1}X), & X\in\mathbb S_{++}^d,\\[6pt]
+\infty, & X\in\partial\mathbb S_+^d.
\end{cases}
\]
Since
\(X\mapsto -\log\det X\) is convex on \(\mathbb S_{++}^d\), it follows that \(f\)
is convex on \(\mathbb S_+^d\) in the extended-value sense. Let
\[
G(b,\Sigma):=
\begin{cases}
\lambda_1\,f(\Sigma_1/\lambda_1), & \lambda_1>0,\\[4pt]
0, & (\lambda_1,\Sigma_1)=(0,0),\\[4pt]
+\infty, & \lambda_1=0,\ \Sigma_1\neq 0.
\end{cases}
\]
Then \(G\) is the perspective of \(f\), hence it is jointly convex on
\([0,\infty)\times\mathbb S_+^d\). Since \(\bar\Sigma_2^{-1}\in\mathbb S_{++}^d\), we have \(\lambda_1(\Sigma_1)\ge 0\) and
\(\lambda_1(\Sigma_1)=0\) if and only if \(\Sigma_1=0\). Moreover,
\(\Sigma_1\mapsto (\lambda_1(\Sigma_1),\Sigma_1)\) is affine. Therefore both
$
\Sigma_1\mapsto \psi(\lambda_1(\Sigma_1))
\qquad\text{and}\qquad
\Sigma_1\mapsto G(\lambda_1(\Sigma_1),\Sigma_1)
$
are convex on \(\mathbb S_+^d\). Finally, one checks directly from the three cases in the definition of \(\ell^{\mathrm{tr}}\) that
$
\ell^{\mathrm{tr}}(t, x, \Sigma_1)=\psi(\lambda_1(\Sigma_1))+G(\lambda_1(\Sigma_1),\Sigma_1),
\ \Sigma_1\in\mathbb S_+^d.
$
Hence \(\ell^{\mathrm{tr}}\) is convex on \(\mathbb S_+^d\).
\end{proof}

The first step to establish the strong duality result is to prove tightness.
For this purpose, we introduce an enlarged canonical space
$
\bar\Omega := C\big([0,1];\,\mathbb R^d \times \mathbb R^{d^2}\big),
$
endowed with its Borel $\sigma$-field $\bar{\mathcal F}_1$ and the canonical filtration
$\bar{\mathbb F}=(\bar{\mathcal F}_t)_{0\le t\le1}$.
We denote by $(X,A)$ the canonical process on $\bar\Omega$, where $X$ is a $d$-dimensional process
and $A$ is is a $d^2$-dimensional process.
We consider probability measures $\bar{\mathbb P}$ on $(\bar\Omega,\bar{\mathcal F}_1)$ such that
$X$ is a $\bar{\mathbb F}$-continuous local martingale under $\bar{\mathbb P}$ characterized by $A$.
Moreover, we assume that the process $A$ is $\bar{\mathbb P}$-a.s. absolutely continuous with
respect to the Lebesgue measure.
In this case, we define the instantaneous covariance process
$\Sigma_1=(\Sigma_1(t))_{0\le t\le1}$ by
$
\Sigma_1(t)
:= \limsup_{n\to\infty} n\bigl(A_t - A_{t-1/n}\bigr),
\ dt\times d\bar{\mathbb P}\text{-a.e.}
$
and assume that $\Sigma_1(t)\in\mathbb{S}_d^{
+}$ almost everywhere.
We denote by $\bar{\mathcal P}$ the collection of all probability measures
$\bar{\mathbb P}$ on $(\bar\Omega,\bar{\mathcal F}_1)$ satisfying the above conditions.
For $\mu_0,\mu_1\in\mathcal M(\mathbb R^d)$, we further define
$
\bar{\mathcal P}(\mu_0)
:= \bigl\{ \bar{\mathbb P}\in\bar{\mathcal P} : \bar{\mathbb P}\circ X_0^{-1}=\mu_0 \bigr\}$,
$\bar{\mathcal P}(\mu_0,\mu_1)
:= \bigl\{ \bar{\mathbb P}\in\bar{\mathcal P}(\mu_0) :
\bar{\mathbb P}\circ X_1^{-1}=\mu_1 \bigr\}.
$
Finally, for $\bar{\mathbb P}\in\bar{\mathcal P}$, we define the lifted cost functional
\[
\bar J(\bar{\mathbb P})
:=
\mathbb E^{\bar{\mathbb P}}\!\left[
\int_0^1
\ell^{\mathrm{tr}}\bigl(t, X_t, \Sigma_{1,t}^{\bar{\mathbb P}}\bigr)\,dt
+
\xi\bigl(X_{\cdot \wedge 1}\bigr)
\right].
\]
where $\ell^{\mathrm{tr}}:[0,1]\times \mathbb R^d \times \mathbb{S}_d^{+}\to\mathbb R$ and \(\xi:C([0,1];\mathbb R^d)\to\mathbb R\) are the costs introduced in the primal problem.
\begin{lemma}[Young / entropy inequality]\label{lem:young}
For all $x\ge 0$ and all $y\in\R$,
\begin{equation}\label{eq:young}
xy \le x\log x + e^{y-1}.
\end{equation}
\end{lemma}
\begin{proof}
For fixed $y$, the convex function $x\mapsto x\log x - xy$ attains its minimum at $x=e^y$,
with value $-e^{y-1}$. Rearranging yields \eqref{eq:young}.
\end{proof}

\begin{remark}[Control of the positive part]\label{rem:positivepart}
For all $x\ge 0$, one has $x\log x\ge -1/e$. Consequently, for every $n$ and every $t$ and $x$,
$
(\lambda_1^n(t,x)\log \lambda_1^n(t,x))_+ \le \lambda_1^n(t,x)\log \lambda_1^n(t,x) + \frac1e.
$
In particular, under \eqref{eq:entropy_integrability},
\[
\sup_{n\ge 1}\ \mathbb E^{\mathbb{P}^n}\!\left[\int_0^1 (\lambda_1^n(t,X_t)\log \lambda_1^n(t,X_t))_+\,dt\right]\le C+\frac1e.
\]
\end{remark}

\begin{lemma}[Small-set control]\label{lem:smallset}
Let $A\subset[0,1]$ be measurable with Lebesgue measure $|A|=\delta\in(0,1)$.
Under \eqref{eq:entropy_integrability},
\[
\sup_{n\ge 1}\ \mathbb E^{\mathbb{P}^n}\!\left[\int_A \lambda_1^n(t,X_t)\,dt\right]
\le \frac{C+\frac{2}{e}}{\log(1/\delta)}.
\]
\end{lemma}

\begin{proof}
Fix $n$ and $\delta\in(0,1)$. Apply Lemma~\ref{lem:young} with
$x=\lambda_1^n(t, x)\ge 0$ and $y=\log(1/\delta)$ to obtain, for each $t$,
\[
\lambda_1^n(t, x)\log(1/\delta)\le \lambda_1^n(t, x)\log(\lambda_1^n(t, x)) + e^{\log(1/\delta)-1}
= \lambda_1^n(t, x)\log(\lambda_1^n(t, x)) + \frac{1}{e\delta}.
\]
Since $u\le u_+$ for all $u\in\mathbb R$, we may bound $\lambda_1^n(t, x)\log(\lambda_1^n(t, x))\le (\lambda_1^n(t, x)\log \lambda_1^n(t, x))_+$,
hence
\[
\lambda_1^n(t, x)\log(1/\delta)\le (\lambda_1^n(t, x)\log \lambda_1^n(t, x))_+ + \frac{1}{e\delta}.
\]
Integrate over $A$:
\[
\log(1/\delta)\int_A \lambda_1^n(t, x)\,dt
\le \int_A (\lambda_1^n(t, x)\log \lambda_1^n(t, x))_+\,dt + \int_A \frac{1}{e\delta}\,dt
= \int_A (\lambda_1^n(t, x)\log \lambda_1^n(t, x))_+\,dt + \frac{1}{e},
\]
because $|A|=\delta$. Divide by $\log(1/\delta)>0$:
\[
\int_A \lambda_1^n(t, x)\,dt
\le \frac{1}{\log(1/\delta)}\int_A (\lambda_1^n(t, x)\log \lambda_1^n(t, x))_+\,dt
+ \frac{1}{e\log(1/\delta)}.
\]
Since $(\lambda_1^n(t, x)\log \lambda_1^n(t, x))_+\ge 0$, we have
\[
\int_A (\lambda_1^n(t, x)\log \lambda_1^n(t, x))_+\,dt \le \int_0^1 (\lambda_1^n(t, x)\log \lambda_1^n(t, x))_+\,dt.
\]
Take expectations and use Remark~\ref{rem:positivepart}: 
\begin{align*}
\mathbb E^{\mathbb{P}^n}\!\left[\int_A \lambda_1^n(t, X_t)\,dt\right]
\le & \frac{1}{\log(1/\delta)}\,
\mathbb E^{\mathbb{P}^n}\!\left[\int_0^1 (\lambda_1^n(t, X_t)\log \lambda_1^n(t, X_t))_+\,dt\right]
+ \frac{1}{e\log(1/\delta)} \\
\le & \frac{C+\frac{1}{e}}{\log(1/\delta)}+\frac{1}{e\log(1/\delta)}.
\end{align*}
Thus
\[
\mathbb E^{\mathbb{P}^n}\!\left[\int_A \lambda_1^n(t, X_t)\,dt\right]\le \frac{C+\frac{2}{e}}{\log(1/\delta)}.
\]
Taking $\sup_n$ finishes the proof.
\end{proof}

\begin{proposition}[Tightness of $\{{\rm Law}(X^n)\}$]
\label{prop:tightness_Xt}

Assume:
\begin{enumerate}
\item[{\rm (i)}] $\{{\rm Law}(X^n_0)\}$ is tight in \(\mathcal M_1(\mathbb R^d)\);
\item[{\rm (ii)}] the entropy bound \eqref{eq:entropy_integrability} holds.
\end{enumerate}
Then $\{{\rm Law}(X^n)\}$ is tight in $\mathcal M_1\bigl(C([0,1];\mathbb R^d)\bigr)$.
\end{proposition}

\begin{proof}
We use Aldous' criterion for tightness in $\mathcal M_1\bigl(D([0,1];\mathbb R^d)\bigr)$.
Tightness of $\{{\rm Law}(X^n_0)\}$ is assumed, so it remains to check the increment condition:
for every $\varepsilon>0$,
\begin{equation}\label{eq:aldous}
\lim_{\delta\downarrow 0}\ \sup_{n\ge 1}\ \sup_{\tau\le 1-\delta}
\mathbb P^n\!\left(\|X_{\tau+\delta}^n - X_\tau^n\|>\varepsilon\right)=0,
\end{equation}
where the inner supremum is taken over all stopping times $\tau$ bounded by $1-\delta$.

Fix $\varepsilon>0$ and $\delta\in(0,1)$. Let $\tau$ be a stopping time with $\tau\le 1-\delta$.
Using the $d$--dimensional It\^o isometry (increment form),
\[
\begin{aligned}
\mathbb E^{\mathbb{P}^n}\!\left[\|X_{\tau+\delta}^n - X_\tau^n\|^2\right]
&=
\mathbb E^{\mathbb{P}^n}\!\left[
\left\|
\int_\tau^{\tau+\delta}\sigma_1^n(s, X_s)\,dW_s^n
\right\|^2
\right] \\
&=
\mathbb E^{\mathbb{P}^n}\!\left[
\int_\tau^{\tau+\delta}
\|\sigma_1^n(s, X_s)\|_F^2\,ds
\right] \\
&\leq
d\,M\,
\mathbb E^{\mathbb{P}^n}\!\left[
\int_\tau^{\tau+\delta}
\lambda_1^n(s, X_s)\,ds
\right].
\end{aligned}
\]
where in the last inequality we used the definition of $\lambda$ and the boundedness of $\bar\Sigma_2$.
For each $\omega$, the random set $A(\omega)=[\tau(\omega),\tau(\omega)+\delta]$ has Lebesgue measure $\delta$,
so Lemma~\ref{lem:smallset} (whose bound depends only on $\delta$) yields
\[
\sup_{n\ge 1}\ \sup_{\tau\le 1-\delta}\ 
\mathbb E^{\mathbb{P}^n}\!\left[\|X_{\tau+\delta}^n - X_\tau^n\|^2\right]
\le d\,M\,\frac{C+\frac{2}{e}}{\log(1/\delta)}.
\]
By Chebyshev's inequality,
$
\mathbb P^n\!\left(\|X_{\tau+\delta}^n - X_\tau^n\|>\varepsilon\right)
\le \frac{1}{\varepsilon^2}\,
\mathbb E^{\mathbb{P}^n}\!\left[\|X_{\tau+\delta}^n - X_\tau^n\|^2\right],
$
hence
\[
\sup_{n\ge 1}\ \sup_{\tau\le 1-\delta}\ 
\mathbb P^n\!\left(\|X_{\tau+\delta}^n - X_\tau^n\|>\varepsilon\right)
\le \frac{1}{\varepsilon^2}\cdot d\,M\,\frac{C+\frac{2}{e}}{\log(1/\delta)}.
\]
Since $\log(1/\delta)\to\infty$ as $\delta\downarrow 0$, the right-hand side tends to $0$,
which proves \eqref{eq:aldous}. Therefore, Aldous' criterion implies tightness of
$\{{\rm Law}(X^n)\}$  in $\mathcal M_1\bigl(C([0,1];\mathbb R^d)\bigr)$.
\end{proof}

\begin{proposition}[Tightness of $A_t^n$]
\label{prop:Tightness_At}
Let $(\Omega,\mathcal F,(\mathcal F_t)_{t\in[0,1]},\mathbb{P}^n)$ be probability spaces and
let $\Sigma_1^n=(\Sigma_{1,t}^n)_{t\in[0,1]}$ be $\mathbb S_+^d$--valued progressively measurable processes.
Then the laws of $(A_t^n)_{n\ge 1}$ form a tight family in
$\mathcal M_1\bigl(C([0,1];\mathbb S_+^d)\bigr)$,
equipped with the uniform topology induced by any matrix norm.
\end{proposition}

\begin{proof}
We use the Arzel\`a--Ascoli compactness criterion.

By \eqref{eq:entropy_integrability} and the de la Vall\'ee--Poussin criterion,
\(\{\lambda_1^n\}_{n\ge1}\) is uniformly integrable on \([0,1]\times\Omega\), i.e.
for every \(\varepsilon>0\) there exists \(K>0\) such that
\begin{equation}\label{eq:UI-short}
\sup_n \E^{\mathbb{P}^n}\int_0^1 \lambda_1^n(t, X_t)\mathbf 1_{\{\lambda_1^n(t, X_t)>K\}}\,dt<\varepsilon.
\end{equation}
In particular,
\begin{equation}\label{eq:L1-short}
\sup_n \E^{\mathbb{P}^n}\int_0^1 \lambda_1^n(t, X_t)\,dt<\infty.
\end{equation}

Since \(A_t^n-A_s^n=\int_s^t \Sigma_1^n(u, x)\,du\succeq0\) for \(s\le t\), and any norm on
\(\mathbb S^d\) is bounded by the trace on \(\mathbb S_+^d\), there is a constant \(C>0\) such that
$
\|A_t^n-A_s^n\|
\le C\int_s^t \tr(\Sigma_1^n(u,x))\,du.
$
Using \(\tr(\Sigma_1^n(u,x))\le C\lambda_1^n(u,x)\), we obtain
\begin{equation}\label{eq:key-estimate}
\|A_t^n-A_s^n\|
\le C\int_s^t \lambda_1^n(u,x)\,du,
\qquad 0\le s\le t\le1.
\end{equation}

We first prove uniform boundedness in probability. Since \(t\mapsto A_t^n\) is increasing in the Loewner order,
\[
\|A^n\|_\infty=\sup_{t\in[0,1]}\|A_t^n\|
\le C\,\tr(A_1^n)
\le C\int_0^1 \lambda_1^n(t,x)\,dt.
\]
Hence, by Markov's inequality and \eqref{eq:L1-short},
$
\sup_n P_n(\|A^n\|_\infty>S)\xrightarrow[S\to\infty]{}0.
$

Next, let \(\omega(f,\delta):=\sup_{|t-s|\le\delta}\|f(t)-f(s)\|\). From \eqref{eq:key-estimate},
\[
\omega(A^n,\delta)\le C\,\sup_{|t-s|\le\delta}\int_s^t \lambda_1^n(u,x)\,du.
\]
Fix \(\eta>0\) and \(\varepsilon>0\). Choose \(K\) as in \eqref{eq:UI-short}. Then for \(|t-s|\le\delta\),
\[
\int_s^t \lambda_1^n(u,x)\,du
\le K\delta+\int_0^1 \lambda_1^n(u,x)\mathbf 1_{\{\lambda_1^n(u,x)>K\}}\,du.
\]
Choose \(\delta>0\) so that \(CK\delta\le \eta/2\). Then
\[
\mathbb P^n\big(\omega(A^n,\delta)>\eta\big)
\le
\mathbb P^n\!\left(
\int_0^1 \lambda_1^n(u,X_u)\mathbf 1_{\{\lambda_1^n(u,X_u)>K\}}\,du>\frac{\eta}{2C}
\right),
\]
and therefore, by Markov's inequality and \eqref{eq:UI-short},
\[
\sup_n \mathbb P^n\big(\omega(A^n,\delta)>\eta\big)
\le \frac{2C}{\eta}\,
\sup_n \E^{\mathbb{P}^n}\int_0^1 \lambda_1^n(u,X_u)\mathbf 1_{\{\lambda_1^n(u,X_u)>K\}}\,du
\le \frac{2C}{\eta}\varepsilon.
\]
Since \(\varepsilon\) is arbitrary,
$
\lim_{\delta\downarrow0}\sup_n \mathbb P^n\big(\omega(A^n,\delta)>\eta\big)=0.
$
We now conclude by Arzel\`a--Ascoli. Fix \(\epsilon>0\). Choose \(S>0\) such that      
$
\sup_n \mathbb P^n(\|A^n\|_\infty>S)<\epsilon/2.
$
Choose \(\eta_m\downarrow0\) and, for each \(m\), \(\delta_m>0\) such that
$
\sup_n \mathbb P^n\big(\omega(A^n,\delta_m)>\eta_m\big)\le \epsilon\,2^{-(m+1)}.
$

Then
$
\mathcal K
:=
\Big\{
f\in C([0,1];\mathbb S^d):
\|f\|_\infty\le S,\ \omega(f,\delta_m)\le \eta_m\ \forall m
\Big\}
$
has compact closure in \(C([0,1];\mathbb S^d)\) by the deterministic Arzel\`a--Ascoli theorem, and
\[
\mathbb P^n(A^n\notin\overline{\mathcal K})
\le
\mathbb P^n(\|A^n\|_\infty>S)
+\sum_{m\ge1}\mathbb P^n\big(\omega(A^n,\delta_m)>\eta_m\big)
<\epsilon.
\]
Thus the laws of \(A^n\) are tight in $\mathcal M_1\bigl(C([0,1];\mathbb R^d)\bigr)$.
\end{proof}

\begin{corollary}
[Tightness of $\bar{\mathbb P}^n$]
\label{thm:tightness}
The sequence $(\bar{\mathbb P}^n)_{n\ge1}$ is tight on $\bar\Omega$.
\end{corollary}

\begin{proof}
Let $\bar{\mathbb P}^n := \mathrm{Law}(X^n,A^n)$ be the probability measure induced by $(X^n,A^n)$ on $\bar\Omega$.
Since the sequence $(\mathrm{Law}(X^n))_{n\ge1}$ is tight in $\mathcal M_1\bigl(C([0,1];\mathbb R^d)\bigr)$ by Proposition \ref{prop:tightness_Xt} and $(\mathrm{Law}(A^n))_{n\ge1}$ is tight in $\mathcal M_1\bigl(C([0,1];\mathbb S_+^d)\bigr)$ by Proposition \ref{prop:Tightness_At}, it follows that
the sequence $(\bar{\mathbb P}^n)_{n\ge1}$ is tight on $\bar\Omega$.
\end{proof}

\begin{lemma}
\label{lem:J_lsc}
The function $\bar J$ is lower semicontinuous on $\overline{\mathcal{P}}$.
\end{lemma}

\begin{proof}
Let \(\bar{\mathbb P}_n\Rightarrow\bar{\mathbb P}_0\) on
$
\bar\Omega.
$
We want to prove
$
\bar J(\bar{\mathbb P}_0)
\le
\liminf_{n\to\infty}\bar J(\bar{\mathbb P}_n).
$

Fix \(\varepsilon\in(0,1)\) and define
$
F_\varepsilon(x,a)
:=
\int_0^{1-\varepsilon}
\ell^{\mathrm{tr}}
\left(
s,x_s,\frac{a_{s+\varepsilon}-a_s}{\varepsilon}
\right)ds .
$
Let
$
\omega_x(\varepsilon)
:=
\sup_{|t-s|\le\varepsilon}|x_t-x_s|.
$
For \(R>0\) and \(\delta>0\), choose continuous cutoffs
\(\beta_R,\gamma_{\varepsilon,\delta}:C([0,1];\mathbb R^d)\to[0,1]\) such that
$
\beta_R(x)=1 \text{ if } \|x\|_\infty\le R$
and $\beta_R(x)=0 \text{ if } \|x\|_\infty\ge R+1,
$
with \(\beta_R\uparrow1\) as \(R\to\infty\), and
$
\gamma_{\varepsilon,\delta}(x)=1$ if $ \omega_x(\varepsilon)\le\delta$;
$
\gamma_{\varepsilon,\delta}(x)=0
$ if $\omega_x(\varepsilon)\ge2\delta.
$
Set
$
\chi_{R,\varepsilon,\delta}:=\beta_R,\gamma_{\varepsilon,\delta}.
$

We first record the pathwise estimate. Let
$
a_t=\int_0^t \eta_r\,dr.
$
On the set where \(\chi_{R,\varepsilon,\delta}(x)>0\), we have
$
\|x\|_\infty<R+1$
and $
\omega_x(\varepsilon)<2\delta.
$
Hence, for \(r\in[s,s+\varepsilon]\),
$
|r-s|\le\varepsilon$, 
$
|x_r-x_s|\le2\delta$
and $
|x_r|\vee |x_s|\le R+1.
$
By convexity of
\(\Sigma\mapsto\ell^{\mathrm{tr}}(s,x_s,\Sigma)\), Jensen's inequality gives
\[
\ell^{\mathrm{tr}}
\left(
s,x_s,\frac1\varepsilon\int_s^{s+\varepsilon}\eta_r\,dr
\right)
\le
\frac1\varepsilon
\int_s^{s+\varepsilon}
\ell^{\mathrm{tr}}(s,x_s,\eta_r)\,dr .
\]
By Assumption \ref{ass:ref_process_duality}, for every \(R>0\), there exists
\(\Delta_{\ell,R}(\varepsilon,\delta)\to0\) as
\((\varepsilon,\delta)\to(0,0)\), such that
$
\ell^{\mathrm{tr}}(s,x,\Sigma)
\le
\bigl(1+\Delta_{\ell,R}(\varepsilon,\delta)\bigr)
\ell^{\mathrm{tr}}(t,y,\Sigma)
+
\Delta_{\ell,R}(\varepsilon,\delta)
$
whenever
\[
|s-t|\le\varepsilon,\qquad |x-y|\le\delta,\qquad
|x|\vee |y|\le R,\qquad \Sigma\in\mathbb S^d_+.
\]
We thus have 
$
\ell^{\mathrm{tr}}(s,x_s,\eta_r)
\le
\bigl(1+\Delta_{\ell,R+1}(\varepsilon,2\delta)\bigr)
\ell^{\mathrm{tr}}(r,x_r,\eta_r)
+
\Delta_{\ell,R+1}(\varepsilon,2\delta).
$
Therefore, after integrating in \(s\) and using Fubini together with
\(\ell^{\mathrm{tr}}\ge0\),
\[
F_\varepsilon(x,a)
\le
\bigl(1+\Delta_{\ell,R+1}(\varepsilon,2\delta)\bigr)
\int_0^1
\ell^{\mathrm{tr}}(t,x_t,\eta_t)\,dt
+
\Delta_{\ell,R+1}(\varepsilon,2\delta).
\]
Consequently, for every admissible \((x,a)\),
\[
\int_0^1
\ell^{\mathrm{tr}}(t,x_t,\dot a_t)\,dt
\ge
\chi_{R,\varepsilon,\delta}(x)
\left[
\frac{F_\varepsilon(x,a)}
{1+\Delta_{\ell,R+1}(\varepsilon,2\delta)}
-
\Delta_{\ell,R+1}(\varepsilon,2\delta)
\right].
\]

Next, \(F_\varepsilon\) is lower semicontinuous on \(\bar\Omega\). Indeed, if
\((x^m,a^m)\to(x,a)\) uniformly, then we have 
$
(a^m_{s+\varepsilon}-a^m_s)/{\varepsilon}
\longrightarrow
(a_{s+\varepsilon}-a_s)/{\varepsilon}
$
for every \(s\). Since \(\ell^{\mathrm{tr}}\) is lower semicontinuous, Fatou's
lemma yields
$
F_\varepsilon(x,a)
\le
\liminf_{m\to\infty}F_\varepsilon(x^m,a^m).
$
Since \(\chi_{R,\varepsilon,\delta}\) is continuous and \(F_\varepsilon\ge0\),
the product \(\chi_{R,\varepsilon,\delta}F_\varepsilon\) is also lower
semicontinuous. Hence, by Portmanteau,
\[
\liminf_{n\to\infty}
\mathbb E^{\bar{\mathbb P}_n}
\left[
\chi_{R,\varepsilon,\delta}(X)F_\varepsilon(X,A)
\right]
\ge
\mathbb E^{\bar{\mathbb P}_0}
\left[
\chi_{R,\varepsilon,\delta}(X)F_\varepsilon(X,A)
\right].
\]
\noindent
Using the pathwise estimate under \(\bar{\mathbb P}_n\), we obtain
\[
\liminf_{n\to\infty}
\mathbb E^{\bar{\mathbb P}_n}
\left[
\int_0^1
\ell^{\mathrm{tr}}(t,X_t,\Sigma_t)\,dt
\right]
\ge
\frac{
\mathbb E^{\bar{\mathbb P}_0}
\left[
\chi_{R,\varepsilon,\delta}(X)F_\varepsilon(X,A)
\right]
}{
1+\Delta_{\ell,R+1}(\varepsilon,2\delta)
}
-
\Delta_{\ell,R+1}(\varepsilon,2\delta).
\]

We now pass to the limit in the smoothing parameters. Since \(X\) has
continuous paths,
$
\gamma_{\varepsilon,\delta}(X)\to1
\
\bar{\mathbb P}_0\text{-a.s.}
$
for every fixed \(\delta>0\). Moreover, $
\frac{A_{s+\varepsilon}-A_s}{\varepsilon}
=
\frac1\varepsilon\int_s^{s+\varepsilon}\Sigma_r\,dr
\longrightarrow
\Sigma_s
$
for \(ds\otimes d\bar{\mathbb P}_0\)-a.e. \((s,\omega)\). Therefore, by lower
semicontinuity of \(\ell^{\mathrm{tr}}\) and Fatou's lemma,
\[
\liminf_{\varepsilon\downarrow0}
\mathbb E^{\bar{\mathbb P}_0}
\left[
\chi_{R,\varepsilon,\delta}(X)F_\varepsilon(X,A)
\right]
\ge
\mathbb E^{\bar{\mathbb P}_0}
\left[
\beta_R(X)
\int_0^1
\ell^{\mathrm{tr}}(t,X_t,\Sigma_t)\,dt
\right].
\]
Letting first \(\varepsilon\downarrow0\), then \(\delta\downarrow0\), and using
$
\lim_{\delta\downarrow0}
\limsup_{\varepsilon\downarrow0}
\Delta_{\ell,R+1}(\varepsilon,2\delta)=0,
$
we get
\[
\liminf_{n\to\infty}
\mathbb E^{\bar{\mathbb P}_n}
\left[
\int_0^1
\ell^{\mathrm{tr}}(t,X_t,\Sigma_t)\,dt
\right]
\ge
\mathbb E^{\bar{\mathbb P}_0}
\left[
\beta_R(X)
\int_0^1
\ell^{\mathrm{tr}}(t,X_t,\Sigma_t)\,dt
\right].
\]
Finally, letting \(R\to\infty\) and using \(\beta_R\uparrow1\) together with
\(\ell^{\mathrm{tr}}\ge0\), the monotone convergence theorem gives
\[
\liminf_{n\to\infty}
\mathbb E^{\bar{\mathbb P}_n}
\left[
\int_0^1
\ell^{\mathrm{tr}}(t,X_t,\Sigma_t)\,dt
\right]
\ge
\mathbb E^{\bar{\mathbb P}_0}
\left[
\int_0^1
\ell^{\mathrm{tr}}(t,X_t,\Sigma_t)\,dt
\right].
\]

It remains to handle the payoff term. Since
\(\bar{\mathbb P}_n\Rightarrow\bar{\mathbb P}_0\), the \(X\)-marginals converge
weakly on \(C([0,1];\mathbb R^d)\). Since \(\xi\) is bounded and continuous for
the uniform topology,
$
\mathbb E^{\bar{\mathbb P}_n}
\left[
\xi(X_{1\wedge\cdot})
\right]
\longrightarrow
\mathbb E^{\bar{\mathbb P}_0}
\left[
\xi(X_{1\wedge\cdot})
\right].
$
Combining the two estimates yields
$
\liminf_{n\to\infty}\bar J(\bar{\mathbb P}_n)
\ge
\bar J(\bar{\mathbb P}_0).
$
Thus \(\bar J\) is lower semicontinuous.
\end{proof}
\begin{remark}
The enlargement of the canonical space is necessary in the above proof because the cost
functional depends on the quadratic variation $A=\langle X\rangle$.
If one works on the space carrying only $X$, applying the Portmanteau
theorem would require the map
$
X \longmapsto \int_0^1 \ell^{\mathrm{tr}}(t, X_t, \Sigma_1(t))\,dt
$
to be lower semicontinuous with respect to the uniform topology.
This would in particular require the quadratic variation $A$
to depend continuously on $X$, which is not the case. By enlarging the space and taking $(X,A)$ as canonical coordinates,
the cost can be approximated by functionals that are lower
semicontinuous in $(X,A)$, allowing the lower semicontinuity argument
to go through.
\end{remark}

\begin{proposition}
\label{prop:lifted-space}
\begin{enumerate}[label=\textnormal{(\roman*)}, align=left, leftmargin=*]
\item[{\rm (i)}]
For any probability measure $\mathbb{P} \in \mathcal{P}(\mu_0,\mu_1)$, there exists a probability
measure $\overline{\mathbb{P}} \in \overline{\mathcal{P}}(\mu_0,\mu_1)$ such that
$J(\mathbb{P}) = \overline{J}(\overline{\mathbb{P}})$.

\item[{\rm (ii)}]
Conversely, for any $\overline{\mathbb{P}} \in \overline{\mathcal{P}}$, there exists a probability measure
$\mathbb{P} \in \mathcal{P}(\mu_0,\mu_1)$ such that
\[
J(\mathbb{P}) = \overline{J}(\overline{\mathbb{P}}).
\]
\end{enumerate}
\end{proposition}

\begin{proof}
(i) The first part of the proof is identical to the argument used by Tan and Touzi \cite{TanTouzi2013} in the proof of their Proposition 3.11.
% \textbf{(i)} Let $\mathbb{P}\in\mathcal P(\mu_0,\mu_1)$. By \ref{eq:mart-decomp}, there exists an $\F$-adapted continuous local martingale $M$
% such that $X_t=X_0+M_t$, $\mathbb{P}$-a.s. for all $t$.
% Let $A^\mathbb{P}:=\langle M\rangle$ be the quadratic variation of $M$.
% In particular, for each $t$, the map $\omega\mapsto (X_t(\omega),A_t^\mathbb{P}(\omega))$ is $\F_t$-measurable.
% It then follows that the map
% \[
% \Omega\ni \omega \longmapsto (X(\omega),A^\mathbb{P}(\omega))\in C([0,1],\R^d)\times C([0,1],\mathbb S^d^+)=:\bar\Omega
% \]
% is $\F_1$-measurable (see Chapter 2 of \cite{Billingsley1968}).

% Define $\overline{\mathbb{P}}:=\mathbb{P}\circ (X,A^\mathbb{P})^{-1}$.
% In the enlarged space $(\overline{\Omega}, \overline{\mathcal{F}}_1, \overline{\mathbb{P}})$, 
% the canonical process $X$ is clearly a continuous local martingale 
% characterized by $A^{\mathbb{P}}(X)$. 
% Moreover,
% $
% A^{\mathbb{P}}(X) = A, 
% \quad \overline{\mathbb{P}}\text{-a.s.},
% $
% where $(X, A)$ are canonical processes in $\Omega$. 
% It follows that, on the enlarged space 
% $(\overline{\Omega}, \overline{\mathcal{F}}, \overline{\mathbb{P}})$, 
% $X$ is a continuous local martingale characterized by $A$. 
% Also, $A$ is clearly $\overline{\mathbb{P}}$-a.s. absolutely continuous.
% Hence $\overline{\mathbb{P}}\in \bar{\mathcal P}(\mu_0,\mu_1)$.

% Finally, we obtain
% \[
% J(\overline{\mathbb{P}})=\E^{\overline{\mathbb{P}}}\Big[\int_0^1 \bigl(L(\Sigma_1^{\mathbb P}(t)) + c(t,X_t)\bigr)\,dt\Big]
% =\E^{\mathbb{P}}\Big[\int_0^1 \bigl(L(\Sigma_1^{\mathbb P}(t)) + c(t,X_t)\bigr)\,dt\Big]
% =J(\mathbb{P}).
% \]

(ii) Let $\overline{\mathbb{P}}\in\bar{\mathcal P}(\mu_0,\mu_1)$. We consider the enlarged space $\overline{\Omega}$, and denote by 
$\overline{\mathbb{F}}^{X} = (\overline{\mathcal{F}}^{X}_t)_{0 \le t \le 1}$ 
the filtration generated by process $X$. Under $\overline{\mathbb{P}}$, the canonical process $X$ is a continuous local martingale
with respect to the canonical filtration
$
\bar\F_t:=\sigma(X_u,\bar A_u:\ u\le t).
$
Let $\F^X_t:=\sigma(X_u:\ u\le t)$.
Since $\F^X_t\subset\bar\F_t$ and $X$ is $\F^X$-adapted, $X$ remains a continuous $\F^X$--local martingale under $\overline{\mathbb{P}}$ with same density $\Sigma_s$ of its quadratic variation process.

Finally, since
$
\overline{\mathcal{F}}^{X}_t
=
\mathcal{F}_t \otimes 
\{\varnothing, C([0,1], \mathbb{R}^{d^2})\},
$
$\overline{\mathbb{P}}$ then induces a probability measure $\mathbb{P}$ on
$(\Omega, \mathcal{F}_1)$ by
$
\mathbb{P}[E]
:=
\overline{\mathbb{P}}
\big[E \times C([0,1], \mathbb{R}^{d^2} \big)],
\
\forall E \in \mathcal{F}_1.
$
Moreover, since the local martingale property holds in the filtration
generated by $X$ under $\bar{\mathbb P}$, it also holds under $\mathbb P$
in the canonical filtration.
Thus $X$ is a continuous $\mathbb F^X$-local martingale under $\mathbb P$.
Clearly, $\mathbb{P} \in \mathcal{P}(\mu_0,\mu_1)$ and
$J(\mathbb{P}) = \overline{J}(\overline{\mathbb{P}})$.
\end{proof}

\begin{lemma}
\label{lem:V-lsc}
The map
\[
(\mu_0,\mu_1) \in \mathcal{M}(\mathbb{R}^d)\times \mathcal{M}(\mathbb{R}^d)
\longmapsto V(\mu_0,\mu_1) \in \overline{\mathbb{R}} := \mathbb{R}\cup\{\infty\}
\]
is lower semicontinuous.
\end{lemma}

\begin{proof}
Let $(\mu_0^n)_{n \in \mathbb{N}}$ and $(\mu_1^n)_{n \in \mathbb{N}}$ be sequences in $\mathcal{M}(\mathbb{R}^d)$ converging weakly to $\mu_0$ and $\mu_1$, respectively. We aim to show that:
$
\liminf_{n \to \infty} V(\mu_0^n,\mu_1^n) \ge V(\mu_0,\mu_1).
$

By Corrolary~\ref{thm:tightness} , the sequence $(\overline{\mathbb{P}}_n)_{n \ge 1}$ induced by $(P_n, X, A^{P_n})$ on $(\overline{\Omega}, \overline{\mathcal{F}}_1)$ is tight. Prokhorov's theorem thus yields a subsequence and a probability measure $\overline{\mathbb{P}}$ on $\overline{\Omega}$ such that $\overline{\mathbb{P}}_n \Rightarrow \overline{\mathbb{P}}$ weakly.

By the convexity of $\mathbb{S}_+^d$, for every $0 \le s < t \le 1$, we have $(A^n_t - A^n_s)/(t - s) \in \mathbb{S}_+^d$, $\overline{\mathbb{P}}_n$-a.s. Since the evaluation maps are continuous and $\mathbb{S}_+^d$ is closed, weak convergence guarantees that for all rational $0 \le s < t \le 1$,
$
\frac{A_t - A_s}{t - s} \in \mathbb{S}_+^d, \overline{\mathbb{P}}\text{-a.s.}
$
By the density of $\mathbb{Q}$ in $[0,1]$ and the path continuity of $A$, this property extends $\overline{\mathbb{P}}$-a.s. to all real $0 \le s < t \le 1$.
As established in Appendix~\ref{sec:ac_limit}, $A$ remains absolutely continuous under $\overline{\mathbb{P}}$. Defining the derivative $\Sigma_t = \dot{A}_t$ for a.e. $t \in [0,1]$, the Lebesgue differentiation theorem ensures that $\Sigma_t \in \mathbb{S}_+^d$, $dt \otimes d\overline{\mathbb{P}}$-a.e. Furthermore, by Theorem 12 in \cite{MeyerZheng1984}, $X$ remains a local martingale under $\overline{\mathbb{P}}$ with quadratic variation $A$.
Consequently, $\overline{\mathbb{P}} \in \overline{\mathcal{P}}(\mu_0,\mu_1)$. Combining this with Proposition~\ref{prop:lifted-space} and the lower semicontinuity of $\overline{J}$ (Lemma~\ref{lem:J_lsc}), we obtain:
\[
\liminf_{n \to \infty} V(\mu_0^n,\mu_1^n) = \liminf_{n \to \infty} J(\mathbb{P}^n) = \liminf_{n \to \infty} \overline{J}(\overline{\mathbb{P}}_n) \ge \overline{J}(\overline{\mathbb{P}}).
\]

Applying Proposition~\ref{prop:lifted-space} once more, there exists $\mathbb{P} \in \mathcal{P}(\mu_0,\mu_1)$ such that $\overline{J}(\overline{\mathbb{P}}) = J(\mathbb{P})$. We conclude:
$
\liminf_{n \to \infty} V(\mu_0^n,\mu_1^n) \ge J(\mathbb{P}) \ge V(\mu_0,\mu_1).
$
\end{proof}

\begin{proposition}
\label{prop:existence of minimizers}
For every
$\mu_0, \mu_1 \in \mathcal{M}(\mathbb{R}^d)$ such that
$V(\mu_0,\mu_1) < \infty$, existence holds for the minimization problem
$V(\mu_0,\mu_1)$. Moreover, the set of minimizers
$
\bigl\{ \mathbb{P} \in \mathcal{P}(\mu_0,\mu_1) : J(\mathbb{P}) = V(\mu_0,\mu_1) \bigr\}
$
is a compact set of probability measures on $\Omega$.
\end{proposition}

\begin{proof}
We just let $(\mu_0^n,\mu_1^n) = (\mu_0,\mu_1)$ in the proof of
Lemma~\ref{lem:V-lsc}, then the required existence result is proved by following
the same arguments.
\end{proof}

\begin{lemma}
\label{lem:V_convex}
The map
$
(\mu_0,\mu_1) \mapsto V(\mu_0,\mu_1)
$
is convex.
\end{lemma}
\begin{proof}
The proof follows exactly the same arguments as those
of Tan and Touzi \cite{TanTouzi2013} in the proof of their Lemma 3.15.
\end{proof}

\begin{proof}[Proof of Theorem \ref{theorem:strong duality}]
We follow the proof of Theorem 3.6 in \cite{TanTouzi2013}. We will use also use Lemma 3.5 of \cite{TanTouzi2013}, which states 
\begin{equation}
\mu_0(\phi_0)
=
\inf_{\mathbb{P}\in\mathcal{P}(\mu_0)}
\mathbb{E}^{\mathbb{P}}
\left[
\int_0^1 \ell^{\mathrm{tr}}(t,X_t,\Sigma_{1,t}^{\mathbb P})\,dt
+ \phi_1(X_1)
+ \xi(X_{1\wedge\cdot})\right].
\label{eq:lambda_0}
\end{equation}If
$V(\mu_0,\mu_1)=+\infty$ for every $\mu_1\in\mathcal M(\mathbb R^d)$, then
\eqref{eq:dual} and \eqref{eq:lambda_0} imply
$V(\mu_0,\mu_1)=\mathcal V(\mu_0,\mu_1)=+\infty$, so there is nothing to prove.
Hence we may assume that $V(\mu_0,\mu_1)<+\infty$ for at least one
$\mu_1\in\mathcal M(\mathbb R^d)$.
View $\mathcal M(\mathbb R^d)$ as a subset of
$\overline{\mathcal M}(\mathbb R^d)$ with the subspace topology, and define
$\overline V:\overline{\mathcal M}(\mathbb R^d)\to(-\infty,+\infty]$ by
\begin{equation}\label{eq:extension}
\overline V(\eta):=
\begin{cases}
V(\mu_0,\eta), & \eta\in\mathcal M(\mathbb R^d),\\
+\infty, & \text{otherwise}.
\end{cases}
\end{equation}
Then $\overline V$ is proper, convex, and lower semicontinuous on the locally
convex space $\overline{\mathcal M}(\mathbb R^d)$. By Theorem 2.2.15 and
Lemma 3.2.3 in \cite{DeuschelStroock1989},
\begin{equation}\label{eq:fm-mg}
\overline V(\mu_1)
=\sup_{\phi_1\in C_b(\mathbb R^d)}
\bigl\{\mu_1(-\phi_1)-\overline V^*(-\phi_1)\bigr\},
\qquad \mu_1\in\mathcal M(\mathbb R^d),
\end{equation}
where
$
\overline V^*(\varphi)
:=\sup_{\eta\in\overline{\mathcal M}(\mathbb R^d)}
\bigl\{\eta(\varphi)-\overline V(\eta)\bigr\},
\ \varphi\in C_b(\mathbb R^d).
$
For $\phi_1\in C_b(\mathbb R^d)$, \eqref{eq:extension} gives
\begin{align}
\overline V^*(-\lambda_1)
&=\sup_{\mu_1\in\mathcal M(\mathbb R^d)}
\bigl\{\mu_1(-\phi_1)-V(\mu_0,\mu_1)\bigr\} \nonumber\\
&=\sup_{\mu_1\in\mathcal M(\mathbb R^d)}
\sup_{\mathbb P \in\mathcal P(\mu_0,\mu_1)}
\Bigl\{-\mathbb E^\mathbb P[\phi_1(X_1)]
-\mathbb E^\mathbb{P}\!\Big[\int_0^1 \ell^{\mathrm{tr}}(t,X_t,\Sigma_{1,t}^{\mathbb P})\,dt + \xi(X_{1\wedge\cdot})\Big]\Bigr\}. \nonumber
\end{align}
Since every $\mathbb P\in\mathcal P(\mu_0)$ has terminal law
$\mu_1:=\mathbb P\circ X_1^{-1}\in\mathcal M(\mathbb R^d)$ and belongs to
$\mathcal P(\mu_0,\mu_1)$,
\begin{equation}\label{eq:conj-id-mg}
\overline V^*(-\phi_1)
=-\inf_{\mathbb{P}\in\mathcal P(\mu_0)}
\mathbb E^\mathbb{P}\!\Big[\int_0^1 \ell^{\mathrm{tr}}(t,X_t,\Sigma_{1,t}^{\mathbb P})\,dt
+\phi_1(X_1) + \xi(X_{1\wedge\cdot})\Big].
\end{equation}
By \eqref{eq:lambda_0}, the right-hand side equals $-\mu_0(\lambda_0)$. Hence,
substituting into \eqref{eq:fm-mg}, we obtain the desired duality formula
$
V(\mu_0,\mu_1)
=\sup_{\phi_1\in C_b(\mathbb R^d)}
\bigl\{\mu_0(\phi_0)-\mu_1(\phi_1)\bigr\}.
$
\end{proof}

\appendix

\section{Preservation of absolute continuity under the limit - Proof of Lemma \ref{lem:V-lsc}}
\label{sec:ac_limit}

\begin{proof}
For $\delta>0$ and $a\in C([0,1];\mathbb S_+^d)$, define
\[
\Phi_\delta(a)
:=
\sup\Big\{\sum_{i=1}^m \|a(t_i)-a(s_i)\| :
(s_i,t_i)\ \text{pairwise disjoint},\ 
\sum_{i=1}^m (t_i-s_i)\le \delta\Big\}.
\]
It suffices to show that
$
\Phi_\delta(A)\longrightarrow 0
, \bar{\mathbb P}\text{-a.s. as }\delta\downarrow0,
$
since this is the standard characterization of absolute continuity for continuous paths.

Fix $n$. Up to modifying on a $\bar{\mathbb P}_n$-null set, we may assume that
$A_t=\int_0^t \Sigma_1^n(s, x)\,ds
$
simultaneously for all $t$ in [0,1], by continuity and equality on rational times. Hence, we have for $0\le s\le t\le1$,
$
A_t-A_s=\int_s^t \Sigma_1^n(u)\,du\in\mathbb S_+^d.
$
Since all norms are equivalent on $\mathbb S^d$, enlarging the constant $C$ if necessary,
$
\|A_t-A_s\|
\le C\,\tr\!\left(\int_s^t \Sigma_1^n(u,x)\,du\right)
\le C\int_s^t \lambda_1^n(u,x)\,du.
$
Therefore, for any finite disjoint family $(s_i,t_i)$,
$
\sum_i \|A_{t_i}-A_{s_i}\|
\le C\int_{\cup_i(s_i,t_i)} \lambda_1^n(u,x)\,du,
$
and thus
\begin{equation}\label{eq:Phi_dom_short}
\Phi_\delta(A)
\le
C\sup\Big\{\int_{\mathbb{R}^d} \lambda_1^n(t,x)\,dt:\ E\subset[0,1]\ \text{Borel},\ |E|\le\delta\Big\},
\qquad \bar{\mathbb P}_n\text{-a.s.}
\end{equation}

Fix $\varepsilon,\eta>0$. By uniform integrability of $(\lambda_1^n)$ in \eqref{eq:UI-short}, choose $K>0$ such that
\[
\sup_n \E^{\bar{\mathbb P}_n}\!\left[\int_0^1 \lambda_1^n(t,X_t,)\mathbf 1_{\{\lambda_1^n(t, X_t)>K\}}\,dt\right]
\le \frac{\eta\varepsilon}{2C}.
\]
Then choose $\delta>0$ so that $CK\delta\le \eta\varepsilon/2$. For every Borel set $E\subset[0,1]$ with $|E|\le\delta$,
\[
\int_{\mathbb{R}^d} \lambda_1^n(t,x)\,dt
\le K\delta+\int_0^1 \lambda_1^n(t,x)\mathbf 1_{\{\lambda_1^n(t,x)>K\}}\,dt.
\]
Combining this with \eqref{eq:Phi_dom_short} and Markov's inequality yields
\[
\sup_n \bar{\mathbb P}_n(\Phi_\delta(A)>\varepsilon)
\le
\frac{1}{\varepsilon}
\left(
CK\delta
+
C\sup_n \E^{\bar{\mathbb P}_n}\!\left[\int_0^1 \lambda_1^n(t,X_t)\mathbf 1_{\{\lambda_1^n(t,X_t)>K\}}\,dt\right]
\right)
\le \eta.
\]
Since $\eta$ is arbitrary,
$
\lim_{\delta\downarrow0}\sup_n \bar{\mathbb P}_n(\Phi_\delta(A)>\varepsilon)=0.
$

For fixed $\delta$, the map $a\mapsto \Phi_\delta(a)$ is lower semicontinuous on
$C([0,1];\mathbb S_+^d)$, being the supremum of the continuous maps
$
a\mapsto \sum_i \|a(t_i)-a(s_i)\|.
$
Hence $\{a:\Phi_\delta(a)>\varepsilon\}$ is open, and Portmanteau gives
$
\bar{\mathbb P}(\Phi_\delta(A)>\varepsilon)
\le
\liminf_{n\to\infty}\bar{\mathbb P}_n(\Phi_\delta(A)>\varepsilon).
$
Therefore, for all $\varepsilon>0$,
$
\lim_{\delta\downarrow0}\bar{\mathbb P}(\Phi_\delta(A)>\varepsilon)=0
$
\noindent
Finally, let $\varepsilon_m:=2^{-m}$ and choose $\delta_m\downarrow0$ such that
$
\bar{\mathbb P}(\Phi_{\delta_m}(A)>\varepsilon_m)\le 2^{-m}.
$
By Borel--Cantelli,
$
\Phi_{\delta_m}(A)\to0
, \bar{\mathbb P}\text{-a.s.}
$
Since $\delta\mapsto\Phi_\delta(a)$ is non-decreasing, it follows that
$
\lim_{\delta\downarrow0}\Phi_\delta(A)=0
, \bar{\mathbb P}\text{-a.s.}
$
Hence $A$ is absolutely continuous on $[0,1]$, $\bar{\mathbb P}$-a.s.
\end{proof}

\section{Monotonicity of the dual functional - Proof of (\ref{eq:frac_D})}
\label{sec:monotonicity}

\begin{proof}
We work formally, assuming enough smoothness and decay at infinity to justify all differentiations and integrations by parts.

Fix \(\psi\) and a smooth perturbation \(h\). For \(\varepsilon\) small, let
\[
\phi^\varepsilon := \phi^{\psi+\varepsilon h},
\qquad
u := \left.\frac{d}{d\varepsilon}\right|_{\varepsilon=0}\phi^\varepsilon.
\]
Differentiating
$
\partial_t \phi^\varepsilon + H(\nabla_x^2\phi^\varepsilon)=0$ with $
\phi^\varepsilon(1,\cdot)=\psi+\varepsilon h,
$
at \(\varepsilon=0\), yields the following PDE
$
\partial_t u + D H(\nabla_x^2\phi^\psi)[\nabla_x^2 u]=0,
\
u(1,\cdot)=h.
$
Since \(\Sigma^\psi=\Sigma^*(\nabla_x^2\phi^\psi)\) minimizes the definition of \(H\), the envelope theorem gives
$
D H(M)[N]=\frac12 \Sigma^*(M):N.
$
Therefore \(u\) solves the linear backward equation
\begin{equation}\label{eq:lin-u}
\partial_t u + \frac12 \Sigma^\psi : \nabla_x^2 u =0,
\qquad
u(1,\cdot)=h.
\end{equation}

Now pair \(u\) with \(p^\psi\), which solves
$
\partial_t p^\psi = \frac12 \nabla_x^2 : (\Sigma^\psi p^\psi)$ with initial condition $
p^\psi(0,\cdot)=\mu_0.
$
Using \eqref{eq:lin-u} and integrating by parts,
\begin{align*}
\frac{d}{dt}\int_{\R^d} u(t,x)p^\psi(t,x)\,dx
&=
\int_{\R^d} (\partial_t u)\,p^\psi\,dx
+\int_{\R^d} u\,(\partial_t p^\psi)\,dx \\
&=
-\frac12\int_{\R^d} (\Sigma^\psi:\nabla_x^2u)\,p^\psi\,dx
+\frac12\int_{\R^d} u\,\nabla_x^2:(\Sigma^\psi p^\psi)\,dx \\
&=0.
\end{align*}
Hence the pairing is constant in time, so
$
\int_{\R^d} u(0,x)\,\mu_0(dx)
=
\int_{\R^d} h(x)\,p_1^\psi(x)\,dx.
$

We can now differentiate the dual functional:
\begin{align*}
D\mathcal D(\psi)[h]
&=
\int_{\R^d} u(0,x)\,\mu_0(dx)-\int_{\R^d} h(x)\,\mu_1(dx) \\
&=
\int_{\R^d} h(x)\,\bigl(p_1^\psi(x)-\mu_1(x)\bigr)\,dx.
\end{align*}
So the first variation of \(\mathcal D\) at \(\psi\) is exactly \(p_1^\psi-\mu_1\).
Finally, along the flow
$
\partial_s\psi_s=\log\!\left(\frac{p_1^{\psi_s}}{\mu_1}\right),
$
the chain rule yields
$
\frac{d}{ds}\mathcal D(\psi_s)
=
D\mathcal D(\psi_s)[\partial_s\psi_s]
=
\int_{\R^d}
\bigl(p_1^{\psi_s}-\mu_1\bigr)
\log\!\left(\frac{p_1^{\psi_s}}{\mu_1}\right)\,dx.
$
For any positive densities \(\rho,\mu\) with the same mass,
$
\int_{\R^d}(\rho-\mu)\log\!\left(\frac{\rho}{\mu}\right)\,dx
=
\KL(\rho\mid\mu)+\KL(\mu\mid\rho)\ge 0.
$
Applying this with \(\rho=p_1^{\psi_s}\) and \(\mu=\mu_1\) gives
$
\frac{d}{ds}\mathcal D(\psi_s)\ge 0,
$
which proves the monotonicity.
\end{proof}
\bibliographystyle{unsrt}
\bibliography{ref}

\end{document}